\documentclass[letterpaper,11pt]{article}
 
\usepackage{amsmath,amsthm}
\usepackage{amsfonts}
\usepackage{latexsym}
\usepackage{amssymb,bm}
\usepackage{mathrsfs}
\usepackage{mathtools}
\usepackage[page]{appendix}

\usepackage{changes}

\usepackage{changes}

\usepackage{color}
\usepackage{color}

\newcommand{\bpt}[1]{\mathring{#1}}

\newtheorem{theorem}{Theorem}[section]
\newtheorem{lemma}[theorem]{Lemma}
\newtheorem{corollary}[theorem]{Corollary}
\newtheorem{proposition}[theorem]{Proposition}

\newtheorem{algorithm}{Algorithm}

\newcommand{\inner}[2]{\langle #1,#2\rangle}

\newcommand{\norm}[1]{\|{#1}\|}

\newcommand\set[1]{\{#1\}}

\newcommand{\R}{\mathbb{R}}
\newcommand{\tos}{\rightrightarrows}

\newcommand{\comenta}[1]{}

\newcommand{\HH}{\mathcal{H}}

\newcommand{\mgap}{\vspace{.1in}}

\begin{document}

\title{Iteration complexity of an inexact Douglas-Rachford method and of a Douglas-Rachford-Tseng's F-B four-operator splitting
method for solving monotone inclusions}
\author{}
\author{
    M. Marques Alves
\thanks{
Departamento de Matem\'atica,
Universidade Federal de Santa Catarina,
Florian\'opolis, Brazil, 88040-900 ({\tt maicon.alves@ufsc.br}).
The work of this author was partially supported by CNPq grants no.
406250/2013-8, 306317/2014-1 and 405214/2016-2.}
\and
Marina Geremia
\thanks{
Departamento de Matem\'atica,
Universidade Federal de Santa Catarina,
Florian\'opolis, Brazil, 88040-900 ({\tt marinageremia@yahoo.com.br}).
}
}



\maketitle

\begin{abstract}
In this paper, we propose and study the iteration complexity of an inexact Douglas-Rachford splitting (DRS) method and
a Douglas-Rachford-Tseng's forward-backward (F-B) splitting method for solving two-operator and
four-operator monotone inclusions, respectively. The former method (although based on a slightly different mechanism of iteration) 
is motivated by the recent work of J. Eckstein and W. Yao, in which an inexact DRS method is derived from a special instance
of the hybrid proximal extragradient (HPE) method of Solodov and Svaiter, while the latter one combines the proposed 
inexact DRS method (used as an outer iteration) with a Tseng's F-B splitting type method (used as an inner iteration) for 
solving the corresponding subproblems. We prove iteration complexity bounds for both algorithms in the pointwise (non-ergodic) as
well as in the ergodic sense by showing that they admit two different iterations: one that can be embedded into 
the HPE method, for which the iteration complexity is known since the work of Monteiro and Svaiter, and another one which demands a separate analysis. 
Finally, we perform simple numerical experiments 
to show
the performance of the proposed methods when compared with other existing algorithms.   
\end{abstract}

\pagestyle{plain}


\section{Introduction}
Let $\HH$ be a real Hilbert space. In this paper, we consider the  \emph{two-operator monotone inclusion problem} (MIP) of finding $z$ such that 
\begin{align}
 \label{eq:mip.i}
 0\in A(z)+B(z)
\end{align}
as well as the \emph{four-operator} MIP
\begin{align}
 \label{eq:drti}
 0\in A(z)+C(z)+F_1(z)+F_2(z)
\end{align}
where $A$, $B$ and $C$ are (set-valued) maximal monotone operators on $\HH$, $F_1: D(F_1)\to \HH$ 
is (point-to-point) \emph{Lipschitz continuous}
and $F_2:\HH\to \HH$ is (point-to-point) \emph{cocoercive} (see Section \ref{sec:drt} for the precise statement). 
Problems \eqref{eq:mip.i} and \eqref{eq:drti} appear in different fields of applied mathematics and optimization including convex optimization, signal
processing, PDEs, inverse problems, among others~\cite{bau.com-book,glo.osh.yin-spl.spi16}. 
Under mild conditions on the operators $C$, $F_1$ and $F_2$, problem \eqref{eq:drti} becomes a special instance of \eqref{eq:mip.i} 
with $B:=C+F_1+F_2$. This fact will be important later on in this paper.

In this paper, we propose and study the iteration complexity of an inexact Douglas-Rachford splitting method (Algorithm \ref{inexact.dr}) and of
a Douglas-Rachford-Tseng's forward-backward (F-B) four-operator splitting method (Algorithm \ref{drt}) for solving \eqref{eq:mip.i}
and \eqref{eq:drti}, respectively. 
The former method is inspired and motivated (although based on a slightly different mechanism of iteration) by the recent work of J. Eckstein and W. Yao~\cite{eck.yao-rel.mp17}, while the latter one, which, in particular, will
be shown to be a special instance of the former one, is motivated by some variants of the standard Tseng's F-B
splitting method~\cite{tse-mod.sjco00} recently proposed in the current 
literature~\cite{alv.mon.sva-reg.siam16,arias.davis-half,MonSva10-1}. 
For more detailed information about the contributions of this paper in the light of reference \cite{eck.yao-rel.mp17}, we refer the reader to the first remark after Algorithm \ref{inexact.dr}. Moreover, we mention that 
Algorithm \ref{drt} is a purely primal splitting method for solving the \emph{four-operator} MIP \eqref{eq:drti}, and this seems to be new. The main contributions of this paper will be discussed in Subsection \ref{subsec:main.cont}.

\subsection{The Douglas-Rachford splitting (DRS) method}
 \label{subsec:drsmi}

One of the most popular algorithms for finding approximate solutions of \eqref{eq:mip.i}
is the \emph{Douglas-Rachford splitting} (DRS) \emph{method}. It consists of an iterative procedure in 
which at each iteration the resolvents $J_{\gamma A}=(\gamma A+I)^{-1}$ and
$J_{\gamma B}=(\gamma B+I)^{-1}$ of $A$ and $B$, respectively, are employed separately instead of the 
resolvent $J_{\gamma(A+B)}$ of the full operator $A+B$, which may be expensive to compute numerically. 
An iteration of the method can be described by 
\begin{align}
 \label{eq:dr.ite}
 z_{k}= J_{\gamma A}(2J_{\gamma B}(z_{k-1})-z_{k-1})+z_{k-1}-J_{\gamma B}(z_{k-1})\qquad \forall k\geq 1,
\end{align}
where $\gamma>0$ is a scaling parameter and $z_{k-1}$ is the current iterate. 
Originally proposed in \cite{dou.rac-num.tams56} for solving problems with linear operators, the DRS method was generalized 
in~\cite{lio.mer-spl.sjna79}
for general nonlinear maximal monotone operators, where the formulation \eqref{eq:dr.ite} was first obtained. 
It was proved in \cite{lio.mer-spl.sjna79} that $\{z_k\}$ converges (weakly, in infinite dimensional Hilbert spaces) 
to some $z^*$ such that $x^*:=J_{\gamma B}(z^*)$ is a solution of \eqref{eq:mip.i}. 
Recently, \cite{sva-wea.sjco11} proved the (weak) convergence of the sequence $\{z_k\}$ generated in \eqref{eq:dr.ite} 
to a solution of \eqref{eq:mip.i}.


\subsection{The Rockafellar's proximal point (PP) method}
\label{subsec:ppmi}

The \emph{proximal point}  (PP) \emph{method} is an iterative method for seeking
approximate solutions of the MIP
\begin{align}
 \label{eq:mip.Ti}
 0\in T(z)
\end{align}
where $T$ is a maximal monotone operator on $\HH$ for which the solution set of \eqref{eq:mip.Ti}
is nonempty. In its exact formulation, an iteration of the PP method can
be described by
\begin{align}
 \label{eq:exact.prox}
 z_k= (\lambda_k T+I)^{-1}z_{k-1}\qquad \forall k\geq 1,
\end{align}
where $\lambda_k>0$ is a stepsize parameter and 
$z_{k-1}$ is the current iterate. 
It is well-known that the practical applicability of numerical schemes based on the exact computation of resolvents
of monotone operators strongly depends on strategies that allow for inexact computations. This is the case of 
the PP method \eqref{eq:exact.prox}. In his pioneering work~\cite{roc-mon.sjco76}, Rockafellar proved that 
if, at each iteration $k\geq 1$, $z_k$ is computed satisfying
\begin{align}
 \label{eq:inexact.prox}
 \norm{z_k-(\lambda_k T+I)^{-1}z_{k-1}}\leq e_k,\quad \sum_{k=1}^\infty\,e_k<\infty,
\end{align}
and $\{\lambda_k\}$ is bounded away from zero, then $\{z_k\}$ converges (weakly, in infinite dimensions) to a solution 
of \eqref{eq:mip.Ti}. This result has found important applications in the design and analysis of many practical 
algorithms for solving challenging problems in optimization and related fields.

\subsection{The DRS method is an instance of the PP method (Eckstein and Bertsekas)}
 \label{subsec:drsm.eh.ppm}

In~\cite{eck.ber-dou.mp92}, the DRS method \eqref{eq:dr.ite} was shown to be a special instance of the PP method \eqref{eq:exact.prox} with $\lambda_k\equiv 1$. 
More precisely, it was observed in~\cite{eck.ber-dou.mp92} (among other results) that the sequence
$\{z_k\}$ in \eqref{eq:dr.ite} satisfies
\begin{align}
  \label{eq:car.dr.s}
 z_{k}=(S_{\gamma,\,A,\,B}+I)^{-1}z_{k-1}\qquad \forall k\geq 1,
\end{align}
where $S_{\gamma,\,A,\,B}$ is the maximal monotone operator on $\HH$ 
whose graph is
\begin{align}
 \label{eq:def.s}
 S_{\gamma,\,A,\,B}=\left\{(y+\gamma b,\gamma a+\gamma b)\in \mathcal{H}\times \HH\;|\;b\in B(x),\,a\in A(y),\,
\gamma a+y=x-\gamma b\right\}.
\end{align}
It can be easily checked that $z^*$ is a solution of \eqref{eq:mip.i} if and only if 
$z^*=J_{\gamma B}(x^*)$ for some $x^*$ such that $0\in S_{\gamma,\,A\,B}(x^*)$.
The fact that \eqref{eq:dr.ite} is equivalent to \eqref{eq:car.dr.s} clarifies the proximal nature of the DRS method and allowed \cite{eck.ber-dou.mp92} to obtain inexact and 
relaxed versions of it by alternatively describing \eqref{eq:car.dr.s}
according to the following procedure:
\begin{align}
\label{eq:dr.spp} 
&\mbox{compute}\;\; (x_k,b_k)\;\;\mbox{such that}\;\; b_k\in B(x_k)\;\;\mbox{and}\;\;\gamma b_k+x_k=z_{k-1};\\ 
\label{eq:dr.spp02} 
 &\mbox{compute}\;\; (y_k,a_k)\;\;\mbox{such that}\;\; a_k\in A(y_k)\;\;\mbox{and}\;\;\gamma a_k+y_k=x_k-\gamma b_k;\\
\label{eq:dr.spp03} 
 &\mbox{set}\;\; z_{k}=y_k+\gamma b_k.
\end{align}
%
%
%
%
%

\subsection{The hybrid proximal extragradient (HPE) method of Solodov and Svaiter}
 \label{subsec:hpei}

Many modern inexact versions of the PP method, as opposed to the summable error criterion \eqref{eq:inexact.prox}, use \emph{relative error tolerances} for solving 
the associated subproblems. The first method of this type was proposed in \cite{sol.sva-hyb.jca99}, and subsequently
studied, e.g., in~\cite{MonSva10-1,mon.sva-hpe.siam10,MonSva10-2,sol.sva-hyb.svva99,sol.sva-hyb.jca99,sol.sva-ine.mor00,Sol-Sv:hy.unif}. The key idea consists of decoupling \eqref{eq:exact.prox} in an inclusion-equation
system:
\begin{align}
  \label{eq:dec.prox}
  v\in T(z_+),\quad  \lambda v+z_{+}-z=0,
\end{align} 
where $(z,z_+,\lambda):=(z_{k-1},z_k,\lambda_k)$, and relaxing \eqref{eq:dec.prox} 
within relative error tolerance criteria.
Among these new methods, the \emph{hybrid proximal extragradient} (HPE) 
method of Solodov and Svaiter \cite{sol.sva-hyb.svva99}, which we discuss in details in Subsection \ref{sec:pp}, has been shown to be 
very effective as a framework for the design and analysis 
of many concrete algorithms 
(see, e.g.,~\cite{bot.cse-hyb.nfao15,cen.mor.yao-hyb.jota10,eck.sil-pra.mp13,he.mon-acc.siam16,ius.sos-pro.opt10,lol.par.sol-cla.jca09,mon.ort.sva-imp.coap14,mon.ort.sva-ada.coap16,MonSva10-2,sol.sva-hyb.svva99,sol.sva-ine.mor00,Sol-Sv:hy.unif}).

\subsection{The main contributions of this work}
 \label{subsec:main.cont}

In \cite{eck.yao-rel.mp17},  J. Eckstein and W. Yao proposed and studied the (asymptotic) convergence of an 
inexact version of the DRS method \eqref{eq:dr.ite} 
by applying a special instance of HPE method to the maximal monotone operator given in \eqref{eq:def.s}. 
The resulting algorithm (see \cite[Algorihm 3]{eck.yao-rel.mp17}) allows for inexact computations \emph{in the equation} in \eqref{eq:dr.spp} 
and, in particular, resulted in an inexact version of the ADMM which is suited for large-scale problems, in which fast inner solvers can
be employed for solving the corresponding subproblems (see \cite[Section 6]{eck.yao-rel.mp17}).

In the present work, motivated by \cite{eck.yao-rel.mp17}, we first propose in Section \ref{sec:dr}
an inexact version of the DRS method (Algorithm \ref{inexact.dr}) for solving \eqref{eq:mip.i} in which inexact computations 
are allowed \emph{in both the inclusion and the equation} in \eqref{eq:dr.spp}.
At each iteration, instead of a point in the graph of $B$, Algorithm \ref{inexact.dr} computes a point in the graph of the $\varepsilon$-enlargement 
$B^\varepsilon$ of $B$ (it has the property that $B^{\varepsilon}(z)\supset B(z)$). Moreover, contrary to the reference \cite{eck.yao-rel.mp17}, we study the \emph{iteration complexity} of the proposed method for solving \eqref{eq:mip.i}. 
We show that Algorithm \ref{inexact.dr} admits two type of iterations, one that can be embedded into the HPE method and, on the other hand, another one which demands a separate analysis. We emphasize again that, although motivated by the latter reference, the Douglas-Rachford type method
proposed in this paper is based on a slightly different mechanism of iteration, specially designed
for allowing its iteration complexity analysis (see Theorems \ref{th:idr.main} and \ref{th:idr.main.erg}).

Secondly, in Section \ref{sec:drt}, we consider the four-operator MIP \eqref{eq:drti} and propose and study the iteration complexity
of a Douglas-Rachford-Tseng's F-B splitting type method (Algorithm \ref{drt}) 
which combines Algorithm \ref{inexact.dr} (as an outer iteration) and
a Tseng's F-B splitting type method (Algorithm \ref{hff}) (as an inner iteration) for solving the corresponding subproblems. The
resulting algorithm, namely Algorithm \ref{drt}, has a splitting nature and solves \eqref{eq:drti} without introducing extra variables.
 
Finally, in Section \ref{sec:num}, we perform simple numerical experiments 
to show
the performance of the proposed methods when compared with other existing algorithms.

\subsection{Most related works}
 \label{subsec:mrw}

In \cite{ari-for.opt15}, the relaxed forward-Douglas-Rachford splitting (rFDRS) method was proposed and studied to solve \emph{three-operator MIPs} 
consisting of \eqref{eq:drti} with $C=N_V$, $V$ closed vector subspace, and $F_1=0$. 
Subsequently, among other results, the iteration complexity of the latter method
(specialized to variational problems) was analyzed in \cite{dav-con.sjo15}.
 Problem \eqref{eq:drti} with $F_1=0$ was also considered in
\cite{dav.yin-thr.svva17}, where a three-operator splitting (TOS) method was proposed and its iteration complexity studied. 
On the other hand,
problem \eqref{eq:drti} with $C=N_V$ and $F_2=0$ was studied in \cite{ari-for.jota15}, where the forward-partial inverse-forward splitting
method was proposed and analyzed. In \cite{arias.davis-half}, a Tseng's F-B splitting type method was proposed and analyzed to solve
the special instance of \eqref{eq:drti} in which $C=0$.

The iteration complexity of a relaxed Peaceman-Rachford splitting method for solving \eqref{eq:mip.i} was recently
studied in \cite{mon.che-com.17}. The method of \cite{mon.che-com.17} was shown to be a special instance of a non-Euclidean
HPE framework, for which the iteration complexity was also analyzed in the latter reference (see also \cite{gon.mel.mon-imp.pre16}). 
Moreover, as we mentioned earlier, an inexact version of the DRS method for solving \eqref{eq:mip.i}
was proposed and studied in \cite{eck.yao-rel.mp17}. 


\section{Preliminaries and background materials}
 \label{sec:pre}

\subsection{General notation and $\varepsilon$-enlargements}
\label{sec:gn}

 We denote by $\HH$ a real Hilbert space with inner product $\inner{\cdot}{\cdot}$
and induced norm $\|\cdot\|:=\sqrt{\inner{\cdot}{\cdot}}$ and by $\HH\times \HH$ the product Cartesian endowed with
usual inner product and norm.

A set-valued map $T:\HH\tos \HH$ is said to be a \emph{monotone operator} on $\HH$ if $\inner{z-z'}{v-v'}\geq 0$
for all $v\in T(z)$ and $v'\in T(z')$. Moreover, $T$ is a \emph{maximal monotone operator} if $T$ is monotone and $T=S$ whenever
$S$ is monotone on $\HH$ and $T\subset S$. Here, we identify any monotone operator $T$ with its graph, i.e., we
set $T=\{(z,v)\in \HH\times \HH\,|\,v\in T(z)\}$. The \emph{sum} $T+S$ of two set-valued maps $T,S$ is defined via
 the usual Minkowski sum and
for $\lambda\geq 0$ the operator $\lambda T$ is defined by $(\lambda T)(z)= \lambda T(z):=\{\lambda v\,|\,v\in T(z)\}$. 
The \emph{inverse} of $T:\HH\tos \HH$ is $T^{-1}:\HH\tos \HH$ defined by $v\in T^{-1}(z)$ if and only if $z\in T(v)$. In particular,
$\mbox{zer}(T):=T^{-1}(0)=\{z\in \HH\,|\,0\in T(z)\}$. The \emph{resolvent} of a maximal monotone operator $T$ is
$J_T:=(T+I)^{-1}$, where $I$ denotes the identity map on $\HH$, and, in particular, the following holds: 
$x=J_{\lambda T}(z)$ if and only if $\lambda^{-1}(z-x)\in T(x)$ if and only if $0\in \lambda T(x)+x-z$. We denote by
$\partial_\varepsilon f$ the usual $\varepsilon$-subdifferential of a proper closed convex function $f:\HH\to (-\infty,+\infty]$
and by $\partial f:=\partial f_0$ the Fenchel-subdifferential of $f$ as well. The \emph{normal cone} of a closed convex set
$X$ will be denoted by $N_X$ and by $P_X$ we denote the orthogonal projection onto $X$.

For $T:\HH\tos\HH$ maximal monotone and $\varepsilon\geq 0$, the $\varepsilon$-enlargement~\cite{bur.ius.sva-enl.svva97,leg.the-sub.svva96} of 
$T$
is the operator $T^{\varepsilon}:\HH\tos\HH$ defined by
\begin{align}
 \label{eq:def.teps}
 T^{\varepsilon}(z):=\{v\in \HH\;|\;\inner{z-z'}{v-v'}\geq -\varepsilon\;\;\forall (z',v')\in T\}\quad \forall z\in \HH.
\end{align}
Note that $T(z)\subset T^{\varepsilon}(z)$ for all $z\in \HH$.

The following summarizes some useful properties of $T^{\varepsilon}$ which will be useful
in this paper.
\begin{proposition}
\label{pr:teps}
Let $T, S:\HH\tos \HH$ be set-valued maps.  Then,
\begin{itemize}
\item[\emph{(a)}] if $\varepsilon \leq \varepsilon'$, then
$T^{\varepsilon}(x)\subseteq T^{\varepsilon'}(x)$ for every $x \in \HH$;
\item[\emph{(b)}] $T^{\varepsilon}(x)+S^{\,\varepsilon'}(x) \subseteq
(T+S)^{\varepsilon+\varepsilon'}(x)$ for every $x \in \HH$ and
$\varepsilon, \varepsilon'\geq 0$;
\item[\emph{(c)}] $T$ is monotone if, and only if, $T  \subseteq T^{0}$;
\item[\emph{(d)}] $T$ is maximal monotone if, and only if, $T = T^{0}$;
\end{itemize}
\end{proposition}

Next we present the transportation formula
for $\varepsilon$-enlargements.

\begin{theorem}\emph{(\cite[Theorem 2.3]{bur.sag.sva-enl.col99})}
 \label{th:tf}
  Suppose $T:\HH\tos \HH$ is maximal monotone and
	let $z_\ell, v_\ell\in \HH$, $\varepsilon_\ell, \alpha_\ell\in \R_+$,
	for $\ell=1,\dots, j$, be such that
	 \[
	 v_\ell\in T^{\varepsilon_\ell}(z_\ell),\quad \ell=1,\dots, j,\quad  \sum_{\ell=1}^j\,\alpha_\ell=1,
	\]
	and define
	\[
	 \overline{z}_j:=\sum_{\ell=1}^j\,\alpha_\ell\, z_\ell\,,\quad 
	 \overline{v}_j:=\sum_{\ell=1}^j\,\alpha_\ell\; v_\ell\,,\quad
	 \overline{\varepsilon}_j:=\sum_{\ell=1}^j\,\alpha_\ell \left[\varepsilon_\ell+\inner{z_\ell-\overline{z}_j}
	 {v_\ell-\overline{v}_j}\right].
	\]
	 Then, the following hold:
	 \begin{itemize}
	 \item[\emph{(a)}] $\overline{\varepsilon}_j\geq 0$ and 
	 $\overline{v}_j\in T^{\overline{\varepsilon}_j}(\overline{z}_j)$.
	  \item[\emph{(b)}] If, in addition, $T=\partial f$ for some proper, convex and closed function
		$f$ and $v_\ell\in \partial_{\varepsilon_\ell} f(z_{\ell})$ for $\ell=1,\dots, j$,
		then $\overline{v}_j\in \partial_{\overline{\varepsilon}_j} f(\overline{z}_j)$.
	\end{itemize}
\end{theorem}

\subsection{The hybrid proximal extragradient (HPE) method}
\label{sec:pp}

Consider the \emph{monotone inclusion problem} (MIP) \eqref{eq:mip.Ti}, i.e.,
\begin{align}
 \label{eq:mip.T}
 0\in T(z)
\end{align}
where $T:\HH\tos \HH$ is a maximal monotone operator for which
the solution set $T^{-1}(0)$ of \eqref{eq:mip.T} is nonempty.  

As we mentioned earlier, the proximal point (PP) method of Rockafellar~\cite{roc-mon.sjco76} is one of the most popular algorithms for finding approximate solutions of \eqref{eq:mip.T} and, among the modern inexact versions of the PP method, 
the \emph{hybrid proximal extragradient} (HPE) 
method of \cite{sol.sva-hyb.svva99}, which we present in what follows, has been shown to be 
very effective as a framework for the design and analysis 
of many concrete algorithms 
(see e.g.~\cite{bot.cse-hyb.nfao15,cen.mor.yao-hyb.jota10,eck.sil-pra.mp13,
he.mon-acc.siam16,ius.sos-pro.opt10,lol.par.sol-cla.jca09,mon.ort.sva-imp.coap14,mon.ort.sva-ada.coap16,MonSva10-2,sol.sva-hyb.svva99,sol.sva-ine.mor00,Sol-Sv:hy.unif}). 


\mgap
\mgap

\noindent
\fbox{
\addtolength{\linewidth}{-2\fboxsep}%
\addtolength{\linewidth}{-2\fboxrule}%
\begin{minipage}{\linewidth}
\begin{algorithm}
\label{hpe}
{\bf Hybrid proximal extragradient (HPE) method for \bf{(\ref{eq:mip.T})}}
\end{algorithm}
\begin{itemize}
\item[(0)] Let $z_0\in \HH$ and $\sigma\in [0,1)$ be given and set $j\leftarrow 1$.
\item [(1)] Compute $(\widetilde z_j,v_j,\varepsilon_j)\in \HH\times \HH\times \R_+$
and $\lambda_j>0$ such that
\begin{align}
\label{eq:hpe}
 \begin{aligned}
  v_j\in T^{\varepsilon_j}(\widetilde z_j),\quad  \norm{\lambda_j v_j+\widetilde z_j-z_{j-1}}^2+
 2\lambda_j\varepsilon_j \leq \sigma^2\norm{\widetilde z_j-z_{j-1}}^2.
\end{aligned}
\end{align} 
\item[(2)] Define 
 \begin{align}
  \label{eq:hpe2}
    z_j=z_{j-1}-\lambda_j v_j, 
   \end{align}
	set $j\leftarrow j+1$ and go to step 1.
   \end{itemize}
\noindent
\end{minipage}
} 
\mgap
\mgap

\noindent
{\bf Remarks.} 
\begin{enumerate}
\item If $\sigma=0$ in \eqref{eq:hpe}, then it follows from Proposition \ref{pr:teps}(d)
and \eqref{eq:hpe2} that $(z_+,v):=(z_j,v_j)$
and $\lambda:=\lambda_j>0$ satisfy \eqref{eq:dec.prox}, which means that the HPE method 
generalizes the exact Rockafellar's PP method.
\item Condition \eqref{eq:hpe} clearly
relaxes both the inclusion and the
equation in \eqref{eq:dec.prox} within a relative error criterion. Recall that $T^{\varepsilon}(\cdot)$
denotes the $\varepsilon$-enlargement of $T$ and has the property
that $T^{\varepsilon}(z)\supset T(z)$ (see Subsection \ref{sec:gn} for details). Moreover, in \eqref{eq:hpe2}  an extragradient
step from the current iterate $z_{j-1}$ gives the next iterate $z_{j}$.
\item We emphasize that specific strategies for computing the triple $(\widetilde z_j,v_j,\varepsilon_j)$
as well as the stepsize $\lambda_j>0$ satisfying \eqref{eq:hpe} will depend on the particular instance of
the problem \eqref{eq:mip.T} under consideration. On the other hand, as mentioned before, the HPE method 
can also be used as a framework 
for the design and analysis of concrete algorithms for solving  specific instances
of \eqref{eq:mip.T}
(see, e.g., \cite{eck.sil-pra.mp13,mon.ort.sva-imp.coap14,mon.ort.sva-ada.coap16,MonSva10-1,mon.sva-hpe.siam10,MonSva10-2}).
We also refer the reader to Sections \ref{sec:dr} and \ref{sec:drt}, in this work, for applications of the HPE method
in the context of decomposition/splitting algorithms for monotone inclusions. 
\end{enumerate}

Since the appearance of the paper \cite{mon.sva-hpe.siam10}, we have seen an increasing interest
in studding the \emph{iteration complexity} of
the HPE method and its special instances (e.g., Tseng's forward-backward
splitting method, Korpelevich extragradient method and ADMM~\cite{MonSva10-1,mon.sva-hpe.siam10,MonSva10-2}).
This depends on the following termination
criterion~\cite{mon.sva-hpe.siam10}: given tolerances $\rho, \epsilon>0$, find
$z, v\in \HH$ and $\varepsilon>0$ such that
\begin{align}
  \label{eq:tc}
	 v\in T^{\varepsilon}(z),\quad \norm{v}\leq \rho, \quad \varepsilon\leq \epsilon. 
\end{align}
Note that, by Proposition \ref{pr:teps}(d), if $\rho=\epsilon=0$ in \eqref{eq:tc} then $0\in T(z)$, i.e., 
$z\in T^{-1}(0)$.

We now summarize the main results 
on \emph{pointwise (non ergodic)} and 
\emph{ergodic} iteration complexity~\cite{mon.sva-hpe.siam10} 
of the HPE method that will be used in this paper.
The \emph{aggregate stepsize sequence} 
$\{\Lambda_j\}$ and the \emph{ergodic sequences}
$\{\overline{\widetilde z}_j\}$, $\{\overline{v}_j\}$, 
$\{\overline{\varepsilon}_j\}$ associated
to $\{\lambda_j\}$ and
$\{\widetilde {z}_j\}$,  $\{v_j\}$, and
$\{\varepsilon_j\}$ are, respectively,
\begin{align}
\label{eq:d.eg}
    &\Lambda_j:=\sum_{\ell=1}^j\, \lambda_\ell\,,\\
  \label{eq:d.eg2} 
	&\overline{\widetilde z}_j:= \frac{1}{\;\Lambda_j}\;
   \sum_{\ell=1}^j\,\lambda_\ell\, \widetilde {z}_\ell, \quad 
   \overline{v}_j:= \frac{1}{\;\Lambda_j}\;\sum_{\ell=1}^j\, \lambda_\ell\, v_\ell,\\ 
  \label{eq:d.eg3} 
	&\overline{\varepsilon}_j:=
    \frac{1}{\;\Lambda_j}\;\sum_{\ell=1}^j\,\lambda_\ell \left[\varepsilon_\ell
    +\inner{\widetilde {z}_\ell-
     \overline{\widetilde z}_j}{v_\ell-\overline{v}_j}\right]=
		\frac{1}{\;\Lambda_j}\;\sum_{\ell=1}^j\,\lambda_\ell \left[\varepsilon_\ell
    +\inner{\widetilde {z}_\ell-
     \overline{\widetilde z}_j}{v_\ell}\right].
\end{align}

\begin{theorem}[{\cite[Theorem 4.4(a) and 4.7]{mon.sva-hpe.siam10}}]
  \label{lm:rhpe2}
	Let $\{\widetilde z_j\}$, $\{v_j\}$, etc, be generated by the \emph{HPE method} \emph{(Algorithm \ref{hpe})}
	and let $\{\overline{\widetilde z}_j\}$, $\{\overline{v}_j\}$, etc, be given in \eqref{eq:d.eg}--\eqref{eq:d.eg3}. 
  Let also $d_0$ denote the distance from $z_0$ to $T^{-1}(0)\neq\emptyset$
	and assume that $\lambda_j\geq \underline{\lambda}>0$ for all $j\geq 1$. Then, the following hold:
  \begin{enumerate}
  \item[\emph{(a)}] For any $j\geq 1$,
	  there exists $i\in\{1,\dots,j\}$ 
    such that  
		\begin{align*}
      v_i\in T^{\varepsilon_i}(\widetilde z_i),\quad  \norm{v_i}\leq \dfrac{d_0}{\underline{\lambda}\sqrt{j}}
			\sqrt{\dfrac{1+\sigma}{1-\sigma}},\quad  \varepsilon_i\leq \dfrac{\sigma^2 d_0^2}{2(1-\sigma^2)
			\underline{\lambda}\;j}\,.
    \end{align*}
	 \item[\emph{(b)}]
	  For any $j\geq 1$,
		 \begin{align*}
     \overline{v}_j\in T^{\overline{\varepsilon}_j}(\overline{\widetilde z}_j),\quad 
     \norm{\overline{v}_j}\leq \dfrac{2d_0}{\underline{\lambda}\;j},\quad
     \overline{\varepsilon}_j\leq \dfrac{2(1+\sigma/\sqrt{1-\sigma^2})d_0^2}{\underline{\lambda}\;j}\,. 
    \end{align*}
	\end{enumerate}
\end{theorem}

\noindent
{\bf Remark.}
%
\begin{itemize}
 \item[] The (\emph{pointwise} and \emph{ergodic}) bounds given in (a)
and (b) of Theorem \ref{lm:rhpe2} guarantee, respectively, that  
for given tolerances $\rho,\epsilon>0$, 
the termination criterion \eqref{eq:tc} is satisfied 
in at most
 \begin{align*}
 \mathcal{O}\left(\max\left\{\dfrac{d_0^2}{\underline{\lambda}^2\rho^2},
 \dfrac{d_0^2}{\underline{\lambda}\epsilon}\right\}\right)
 \;\;\mbox{and}\;\;
\mathcal{O}\left(\max\left\{\dfrac{d_0}{\underline{\lambda}\rho},\dfrac{d_0^2}{\underline{\lambda}\epsilon}
\right\}\right)
 \end{align*}
iterations, respectively. We refer the reader to \cite{mon.sva-hpe.siam10} for a complete study of the
iteration complexity of the HPE method and its special instances.
\end{itemize}

\noindent
The proposition below will be useful in the next sections.

\begin{proposition}[{\cite[Lemma 4.2 and Eq. (34)]{mon.sva-hpe.siam10}}]
\label{pr:imp.fact}
Let $\{z_j\}$ be generated by the 
\emph{HPE method (Algorithm \ref{hpe})}. Then, for any $z^*\in T^{-1}(0)$, the sequence $\{\norm{z^*-z_j}\}$ is
nonincreasing.
As a consequence, for every $j \ge 1$, we have
\begin{align}
\label{eq:ineq.zk0}
\|z_j-z_0\| \le 2 d_0,
\end{align}
where $d_0$ denotes the distance of $z_0$ to $T^{-1}(0)$.
\end{proposition}


\subsubsection{A HPE variant for strongly monotone sums}
 \label{subsub:sms}

We now consider the MIP
\begin{align}
 \label{eq:bmu}
 0\in S(z)+B(z)=:T(z)
\end{align}
where the following is assumed to hold:
\begin{itemize}
\item[(C1)] $S$ and $B$ are maximal monotone operators on $\HH$;
\item[(C2)] $S$ is (additionally) $\mu$--strongly monotone for some $\mu>0$, i.e., there exists
$\mu>0$ such that
\begin{align}
 \label{eq:def.strmon}
 \inner{z-z'}{v-v'}\geq \mu\norm{z-z'}^2\qquad \forall v\in S(z),v'\in S(z');
\end{align}
\item[(C3)] the solution set $(S+B)^{-1}(0)$ of \eqref{eq:bmu} is nonempty.
\end{itemize}

 The main motivation to consider the above setting is Subsection \ref{subsec:solsub}, in which
the monotone inclusion \eqref{eq:cff}
is clearly a special instance of \eqref{eq:bmu} with $S(\cdot):=(1/\gamma)(\cdot-\bpt{z})$, which 
is obviously $(1/\gamma)$-strongly maximal monotone on $\HH$.
 
\mgap
The algorithm below was proposed and studied (with a different notation) in \cite[Algorithm 1]{alv.mon.sva-reg.siam16}.

\mgap
\mgap

\noindent
\fbox{
\addtolength{\linewidth}{-2\fboxsep}%
\addtolength{\linewidth}{-2\fboxrule}%
\begin{minipage}{\linewidth}
\begin{algorithm}
\label{shpe}
{\bf A specialized HPE method for solving strongly monotone inclusions}
\end{algorithm}
\begin{itemize}
\item[(0)] Let $z_0\in \HH$ and $\sigma\in [0,1)$ be given and set $j\leftarrow 1$.
\item [(1)] Compute $(\widetilde z_j,v_j,\varepsilon_j)\in \HH\times \HH\times \R_+$
and $\lambda_j>0$ 
such that
\begin{align}
\label{eq:shpe}
 \begin{aligned}
  v_j\in S(\widetilde z_j)+B^{\varepsilon_j}(\widetilde z_j),\quad  
	\norm{\lambda_j v_j+ \widetilde z_j-z_{j-1}}^2+
 2\lambda_j \varepsilon_j \leq \sigma^2\norm{\widetilde z_j-z_{j-1}}^2.
\end{aligned}
\end{align} 
\item[(2)] Define 
 \begin{align}
  \label{eq:shpe2}
    z_j=z_{j-1}-\lambda_j v_j, 
   \end{align}
	set $j\leftarrow j+1$ and go to step 1.
   \end{itemize}
\noindent
\end{minipage}
} 
\mgap
\mgap

%

Next proposition will be useful in Subsection \ref{subsec:solsub}.

\begin{proposition}[{\cite[Proposition 2.2]{alv.mon.sva-reg.siam16}}]
\label{pr:3m}
%
Let $\{\widetilde z_j\}$, $\{v_j\}$ and $\{\varepsilon_j\}$ be generated by \emph{Algorithm \ref{shpe}}, let 
$z^*:=(S+B)^{-1}(0)$ and $d_0:=\norm{z_0-z^*}$. 
Assume that  $\lambda_j\geq \underline{\lambda}>0$ 
for all $j\geq 1$ and define
\begin{align}
\label{eq:def.alpha}
 \alpha :=  \left( \frac{1}{2\underline{\lambda}\mu} +
 \frac{1}{1-\sigma^2} \right)^{-1} \in (0,1).
\end{align}
Then, for all $j\geq 1$, 
\begin{align}
 \label{eq:3m}
 \begin{aligned}
&v_j\in S(\widetilde z_j)+B^{\varepsilon_j}(\widetilde z_j),\\
&\|v_j\| \le \sqrt{\dfrac{1+\sigma}{1-\sigma}}
 \left(\frac{(1-\alpha)^{(j-1)/2}}{\underline{\lambda}} \right) d_0,\\
&\varepsilon_j \le \frac{\sigma^2}{2(1-\sigma^2)}
 \left(\frac{(1-\alpha)^{j-1}}{\underline{\lambda}} \right) d_0^{\,2}.
\end{aligned}
\end{align}
%
%
\end{proposition}

Next section presents one of the main contributions of this paper, namely an inexact Douglas-Rachford type method
for solving \eqref{eq:mip.i} and its iteration complexity analysis.


\section{An inexact Douglas-Rachford splitting (DRS) method and its iteration complexity}
\label{sec:dr}

Consider problem \eqref{eq:mip.i}, i.e., the problem of finding $z\in \HH$ such that
\begin{align}
 \label{eq:mip}
 0\in A(z)+B(z)
\end{align}
where the following hold:
\begin{itemize}
\item[(D1)]  $A$ and $B$ are maximal monotone operators on $\HH$;
\item[(D2)]  the solution set $(A+B)^{-1}(0)$ of \eqref{eq:mip} is nonempty.
\end{itemize} 

In this section, we propose and analyze the iteration complexity of an inexact version 
of the \emph{Douglas-Rachford splitting} (DRS) \emph{method}~\cite{lio.mer-spl.sjna79} for finding approximate solutions of \eqref{eq:mip} according
to the following termination criterion: given tolerances $\rho, \epsilon>0$, find $a,b,x,y\in \HH$ and
$\varepsilon_a,\varepsilon_b\geq 0$ such that
\begin{align}
 \label{eq:def.apsol}
 a\in A^{\varepsilon_a}(y),\;b\in B^{\varepsilon_b}(x),\quad \gamma \norm{a+b}=\norm{x-y}\leq \rho,\; 
\varepsilon_a+\varepsilon_b\leq \epsilon,
\end{align}
where $\gamma>0$ is a scaling parameter. Note that if $\rho=\epsilon=0$ in \eqref{eq:def.apsol}, then
$z^*:=x=y$ is a solution of \eqref{eq:mip}.

As we mentioned earlier, the algorithm below is motivated by \eqref{eq:dr.spp}--\eqref{eq:dr.spp03} as well as by the recent work of Eckstein and Yao~\cite{eck.yao-rel.mp17}.

\mgap
\mgap

\noindent
\fbox{
\addtolength{\linewidth}{-2\fboxsep}%
\addtolength{\linewidth}{-2\fboxrule}%
\begin{minipage}{\linewidth}
\begin{algorithm}
\label{inexact.dr}
{\bf An inexact Douglas-Rachford splitting method for \bf{(\ref{eq:mip})}}
\end{algorithm}
\begin{itemize}
%
\item[(0)] Let $z_0\in \HH$, $\gamma>0$, $\tau_0>0$ and $0<\sigma,\theta<1$  be given and set $k\leftarrow 1$.
\item [(1)] Compute $(x_k,b_k,\varepsilon_{b,\,k})\in \HH\times \HH\times \R_+$ 
such that
\begin{align}
 \label{eq:err.b}
 b_k\in B^{\varepsilon_{b,\,k}}(x_k),\quad \norm{\gamma b_k+x_k-z_{k-1}}^2+2\gamma\varepsilon_{b,\,k}\leq \tau_{k-1}.
\end{align}
\item[(2)]  Compute $(y_k,a_k)\in \HH\times \HH$ such that
 \begin{align}
 \label{eq:err.a}
 \hspace{-3.0cm} a_k\in A(y_k),\quad \gamma a_k+y_k=x_k-\gamma b_k.
\end{align}
\item[(3)]  (3.a) If 
 \begin{align}
 \label{eq:cond.err}
 \hspace{-1.3cm}\norm{\gamma b_k+x_k-z_{k-1}}^2+2\gamma \varepsilon_{b,k}\leq \sigma^2\norm{\gamma b_k+y_k-z_{k-1}}^2,
\end{align}
\hspace{0.9cm}then
\begin{align}
 \label{eq:ext.step}
z_k=z_{k-1}-\gamma(a_k+b_k),\quad \tau_k=\tau_{k-1}\qquad \mbox{[extragradient step]}.
\end{align}
(3.b) Else
\begin{align}
 \label{eq:null.step}
 \hspace{-3.6cm} z_k=z_{k-1},\quad \tau_k=\theta\,\tau_{k-1}\qquad \mbox{[null step]}.
\end{align}
\item[(4)]
	Set $k\leftarrow k+1$ and go to step 1.
   \end{itemize}
\noindent
\end{minipage}
} 
\mgap
\mgap


\noindent
{\bf Remarks.}
\begin{enumerate}

\item  We emphasize that although it has been motivated by \cite[Algorithm 3]{eck.yao-rel.mp17}, Algorithm \ref{inexact.dr} is based on a slightly different mechanism of iteration. Moreover, it also allows for the computation of $(x_k,b_k)$
in \eqref{eq:err.b} in the
 $\varepsilon_{b,k}$-- enlargement of $B$ (it has the property that $B^{\varepsilon_{b,k}}(x)\supset B(x)$ for all $x\in \HH$); this will be crucial for the design and iteration complexity analysis
of the four-operator splitting method of Section \ref{sec:drt}. We also mention that, contrary to this work, no iteration complexity
analysis is performed in \cite{eck.yao-rel.mp17}.


\item Computation of $(x_k,b_k,\varepsilon_{b,\,k})$ satisfying \eqref{eq:err.b} will depend on the particular instance of
the problem \eqref{eq:mip} under consideration. In Section \ref{sec:drt},
we will use Algorithm \ref{inexact.dr} for solving a four-operator splitting monotone inclusion. 
In this setting, at every iteration $k\geq 1$ of Algorithm \ref{inexact.dr}, called an outer iteration, a Tseng's forward-backward (F-B) splitting type method
will be used, as an inner iteration, to solve the (prox) subproblem \eqref{eq:err.b}. 
 %
%
%
 \item Whenever the resolvent $J_{\gamma B}=(\gamma B+I)^{-1}$ is computable, then it follows that
 $(x_k,b_k):=(J_{\gamma B}(z_{k-1}),(z_{k-1}-x_k)/\gamma)$ and $\varepsilon_{b,\,k}:=0$ clearly solve \eqref{eq:err.b}.
In this case, the left hand side of the inequality in \eqref{eq:err.b} is zero and, as a consequence, the
inequality \eqref{eq:cond.err} is always satisfied. In particular, 
\eqref{eq:dr.spp}--\eqref{eq:dr.spp03} hold, i.e., in this case Algorithm \ref{inexact.dr} reduces to the (exact) 
DRS method. 

%
\item In this paper, we assume that the resolvent $J_{\gamma A}=(\gamma A+I)^{-1}$ is computable, which
implies that $(y_k,a_k):=(J_{\gamma A}(x_k-\gamma b_k),(x_k-\gamma b_k-y_k)/\gamma)$ is the demanded pair in \eqref{eq:err.a}.
\item Algorithm \ref{inexact.dr} potentially performs extragradient steps and null steps, 
depending on the condition \eqref{eq:cond.err}. It will be shown in Proposition \ref{pr:eHPE} that iterations corresponding
to extragradient steps reduce to a special instance of the HPE method, in which case pointwise and ergodic iteration complexity results are available in the current literature (see Proposition \ref{pr:c.ihpe}). On the other hand, iterations corresponding
to the null steps will demand a separate analysis (see Proposition \ref{pr:comp.b}).

%
\end{enumerate}

\noindent
As we mentioned in the latter remark, each iteration of Algorithm \ref{inexact.dr} is either an extragradient step
or a null step (see \eqref{eq:ext.step} and \eqref{eq:null.step}).  
This will be formally specified by considering the sets:
\begin{align} 
 \label{eq:def.ab}
 \begin{aligned}
&\mathcal{A}:=\mbox{indexes}\;k\geq 1\;\mbox{for which an extragradient step is executed at the iteration}\; k.\\
&\mathcal{B}:=\mbox{indexes}\;k\geq 1\;\mbox{for which a null step is executed at the iteration}\; k.
\end{aligned}
\end{align}
That said, we let
\begin{align}
 \label{eq:car.a}
 \mathcal{A}=\{k_j\}_{j\in J},\quad J:=\set{j\geq 1\;|\; j\leq \# \mathcal{A}}
\end{align} 
where $k_0:=0$ and $k_0<k_j<k_{j+1}$ for all $j\in J$, and let
$\beta_0:=0$ and 
\begin{align}
  \label{eq:def.betak}
 \beta_k:=\mbox{the number of indexes for which a null step is executed until the iteration}\;k.
\end{align}
Note that direct use of the above definition and \eqref{eq:null.step} yield
\begin{align}
  \label{eq:tau.beta}
 \tau_k=\theta^{\beta_k}\tau_0\quad \forall k\geq 0.
\end{align}


%

%
%
%
%

%
%

%
%

In order to study the \emph{ergodic iteration complexity} of Algorithm \ref{inexact.dr} we also define 
the \emph{ergodic sequences} associated to the sequences $\{x_{k_j}\}_{j\in J}$, $\{y_{k_j}\}_{j\in J}$, 
$\{a_{k_j}\}_{j\in J}$, $\{b_{k_j}\}_{j\in J}$,  
and $\{\varepsilon_{b,\,{k_j}}\}_{j\in J}$, for all $j\in J$, as follows:
\begin{align}
\label{eq:def.erg131} 
&\overline{x}_{k_j}:=\dfrac{1}{j}\sum_{\ell=1}^j x_{k_\ell},\quad 
 \overline{y}_{k_j}:=\dfrac{1}{j}\sum_{\ell=1}^j y_{k_\ell}\,,\\ 
  \label{eq:def.erg130}
 &\overline{a}_{k_j}:=\dfrac{1}{j}\sum_{\ell=1}^j a_{k_\ell}\,,\quad  
  \overline{b}_{k_j}:=\dfrac{1}{j}\sum_{\ell=1}^j b_{k_\ell}\,,\\ 
\label{eq:def.erg132} 
&\overline{\varepsilon}_{a,\,k_j}:=\dfrac{1}{j}\sum_{\ell=1}^j\,
\inner{y_{k_\ell}-\overline{y}_{k_j}}{a_{k_\ell}-\overline{a}_{k_j}}
=\dfrac{1}{j}\sum_{\ell=1}^j\,
\inner{y_{k_\ell}-\overline{y}_{k_j}}{a_{k_\ell}},\,\\
\label{eq:def.erg133}
&\overline{\varepsilon}_{b,\,k_j}:=\dfrac{1}{j}\sum_{\ell=1}^j
\big[\varepsilon_{b,\,k_\ell}+\inner{x_{k_\ell}-\overline{x}_{k_j}}{b_{k_\ell}-\overline{b}_{k_j}}\big]
=\dfrac{1}{j}\sum_{\ell=1}^j
\big[\varepsilon_{b,\,k_\ell}+\inner{x_{k_\ell}-\overline{x}_{k_j}}{b_{k_\ell}}
\big].
\end{align} 

Moreover, the results on iteration complexity of Algorithm \ref{inexact.dr} (pointwise and ergodic) obtained in 
this paper will depend on the following quantity:
\begin{align}
 \label{eq:def.d0}
 d_{0,\,\gamma}:=\mbox{dist}\left(z_0,\mbox{zer}(S_{\gamma, A,B})\right)=
 \min\,\{\norm{z_0-z}\;|\;z\in \mbox{zer}(S_{\gamma, A,B})\}
\end{align}
which measures the quality of the initial guess $z_0$ in Algorithm \ref{inexact.dr} 
with respect to $\mbox{zer}(S_{\gamma, A,B})$, where the operator $S_{\gamma, A,B}$ is 
such that $J_{\gamma B}(\mbox{zer}(S_{\gamma,\,A,\,B}))=(A+B)^{-1}(0)$ (see \eqref{eq:def.s}).

\mgap
In the next proposition, we show that the procedure resulting by selecting the
extragradient steps in Algorithm \ref{inexact.dr} can be embedded into HPE method.

First, we need the following lemma.
\begin{lemma}
 \label{lm:ixk}
Let $\{z_k\}$ be generated by \emph{Algorithm \ref{inexact.dr}} and let the set $J$ be defined in
\eqref{eq:car.a}. Then,
\begin{align}
\label{eq:ixk}
 z_{k_{j-1}}=z_{k_j-1}\qquad \forall j\in J.
\end{align}
\end{lemma}
\begin{proof}
Using \eqref{eq:def.ab} and \eqref{eq:car.a} we have 
$\{k\geq 1\,|\, k_{j-1}<k<k_j\}\subset \mathcal{B}$, for all $j\in J$. Consequently, using
the definition of $\mathcal{B}$ in \eqref{eq:def.ab} and \eqref{eq:null.step} we conclude that 
$z_k=z_{k_{j-1}}$ whenever $k_{j-1}\leq k<  k_j$. 
As a consequence, we obtain that \eqref{eq:ixk} follows
from the fact that $k_{j-1}\leq k_j-1<k_j$.
\end{proof}

\begin{proposition}
 \label{pr:eHPE}
 Let $\{z_k\}$, $\{(x_k,b_k)\}$, $\{\varepsilon_{b,k}\}$ and $\{(y_k,a_k)\}$ be generated 
by \emph{Algorithm \ref{inexact.dr}} and let the operator $S_{\gamma,\,A,\,B}$ be defined in \eqref{eq:def.s}. Define, for all $j\in J$,
\begin{align}
 \label{eq:pr:eHPE.02}
  \widetilde z_{k_j}:= y_{k_j}+\gamma b_{k_j},\quad v_{k_j}:=\gamma(a_{k_j}+b_{k_j}),\quad 
	 \varepsilon_{k_j}:=\gamma \varepsilon_{b,k_j}.
 \end{align}
Then, for all $j\in J$,
\begin{align}
 \begin{aligned}
 \label{eq:pr.eHPE.01}
  &v_{k_j}\in \left(S_{\gamma,\,A,\,B}\right)^{\varepsilon_{k_j}}(\widetilde z_{k_j}),\qquad
  \norm{ v_{k_j}+\widetilde z_{k_j}-z_{k_{j-1}}}^2+2 \varepsilon_{k_j}\leq \sigma^2\norm{\widetilde z_{k_j}-z_{k_{j-1}}}^2,\\[2mm]
	&z_{k_j}=z_{k_{j-1}} - v_{k_j}.
 \end{aligned}
\end{align}
As a consequence, the sequences $\{\widetilde z_{k_j}\}_{j\in J}$, $\{v_{k_j}\}_{j\in J}$, $\{\varepsilon_{k_j}\}_{j\in J}$
and $\{z_{k_j}\}_{j\in J}$ are generated by \emph{Algorithm \ref{hpe}} with $\lambda_j\equiv 1$ for 
solving \eqref{eq:mip.T} with $T:=S_{\gamma,\,A,\,B}$.
\end{proposition}
\begin{proof}
For any $(z',v'):=(y+\gamma b,\gamma a+\gamma b)\in S_{\gamma,\,A,\,B}$ we have, in particular, 
$b\in B(x)$ and $a\in A(y)$ (see \eqref{eq:def.s}). Using these inclusions, the inclusions in \eqref{eq:err.b} and \eqref{eq:err.a}, the monotonicity of the operator $A$ and \eqref{eq:def.teps} with $T=B$ we obtain
\begin{align}
 \label{eq:11}
 \inner{x_{k_j}-x}{b_{k_j}-b}\geq -\varepsilon_{b, k_j},\qquad
 \inner{y_{k_j}-y}{a_{k_j}-a}\geq 0.
\end{align}
Moreover, using the identity in \eqref{eq:err.a} and the corresponding one in \eqref{eq:def.s} we find
\begin{align}
 \label{eq:10}
 (y_{k_j}-y)+\gamma(b_{k_j}-b)=(x_{k_j}-x)-\gamma (a_{k_j}-a).
\end{align}
Using \eqref{eq:pr:eHPE.02}, \eqref{eq:11} and \eqref{eq:10} we have
\begin{align}
 \nonumber
 \inner{\widetilde z_{k_j}-z'}{v_{k_j}-v'}&=\inner{(y_{k_j}+\gamma b_{k_j})-(y+\gamma b)}{(\gamma a_{k_j}+\gamma b_{k_j})-(\gamma a+\gamma b)}\\
 \nonumber
 &=\inner{y_{k_j}-y+\gamma (b_{k_j}-b)}
 {\gamma(a_{k_j}-a)+\gamma(b_{k_j}-b)}\\
\nonumber
&=\gamma\inner{y_{k_j}-y+\gamma (b_{k_j}-b)}
 {a_{k_j}-a}+\gamma\inner{y_{k_j}-y+\gamma (b_{k_j}-b)}
 {b_{k_j}-b}\\
\nonumber
&=\gamma\inner{y_{k_j}-y+\gamma (b_{k_j}-b)}
 {a_{k_j}-a}+\gamma\inner{x_{k_j}-x-\gamma(a_{k_j}-a)}
 {b_{k_j}-b}\\
\nonumber
&=\gamma\inner{y_{k_j}-y}{a_{k_j}-a}+\gamma\inner{x_{k_j}-x}{b_{k_j}-b}\\
\nonumber
&\geq \gamma\inner{x_{k_j}-x}{b_{k_j}-b}\\
\nonumber
&\geq -\varepsilon_{k_j},
\end{align}
which combined with definition \eqref{eq:def.teps} gives the inclusion in \eqref{eq:pr.eHPE.01}.

From \eqref{eq:pr:eHPE.02}, \eqref{eq:ixk}, the identity in \eqref{eq:err.a} and \eqref{eq:cond.err}
we also obtain 
\begin{align*}
\norm{ v_{k_j}+\widetilde z_{k_j}-z_{k_{j-1}}}^2&=
\norm{ \gamma(a_{k_j}+b_{k_j})+(y_{k_j}+\gamma b_{k_j})-z_{k_{j}-1}}^2\\
&=\norm{ (x_{k_j}-y_{k_j})+(y_{k_j}+\gamma b_{k_j})-z_{k_{j}-1}}^2\\
&=\norm{ \gamma b_{k_j}+x_{k_j}-z_{k_{j}-1}}^2\\
&\leq \sigma^2 \norm{ \gamma b_{k_j}+y_{k_j}-z_{k_{j}-1}}^2-2\gamma\varepsilon_{b,k_j}\\
&=\sigma^2 \norm{\widetilde z_{k_j}-z_{k_{j-1}}}^2-2\varepsilon_{k_j},
\end{align*}
which gives the inequality in \eqref{eq:pr.eHPE.01}. To finish the proof of \eqref{eq:pr.eHPE.01},
note that the desired identity in \eqref{eq:pr.eHPE.01} follows from the first one in \eqref{eq:ext.step}, the second one
in \eqref{eq:pr:eHPE.02} and \eqref{eq:ixk}. The last statement of the proposition follows from \eqref{eq:pr:eHPE.02},
\eqref{eq:pr.eHPE.01} and Algorithm \ref{hpe}'s definition.
\end{proof}

\begin{proposition}{\bf (rate of convergence for extragradient steps)}
 \label{pr:c.ihpe}
Let $\{(x_k,b_k)\}$, $\{(y_k,a_k)\}$ and $\{\varepsilon_{b,\,k}\}$ be generated by \emph{Algorithm~\ref{inexact.dr}}
and consider the ergodic sequences defined in \eqref{eq:def.erg131}--\eqref{eq:def.erg133}. 
Let $d_{0,\gamma}$ and the set $J$ be defined
in \eqref{eq:def.d0} and \eqref{eq:car.a}, respectively. 
Then,
\begin{enumerate}
  \item[\emph{(a)}] For any $j\in  J$,
	  there exists $i\in \{1,\dots,j\}$ 
    such that  
		\begin{align}
		  \label{eq:706}
      & a_{k_i}\in A(y_{k_i}),\quad  b_{k_i}\in B^{\varepsilon_{b,\,k_i}}(x_{k_i}),\\  
			\label{eq:306}
			&\gamma\norm{a_{k_i}+b_{k_i}}=\norm{x_{k_i}-y_{k_i}}\leq \dfrac{d_{0,\gamma}}{\sqrt{j}}
			\sqrt{\dfrac{1+\sigma}{1-\sigma}},\\
			\label{eq:206}
			&\varepsilon_{b,\,k_i}\leq \dfrac{\sigma^2 d_{0,\gamma}^{\,2}}{2\gamma(1-\sigma^2) j}\;.
		\end{align}
	 \item[\emph{(b)}]
	  For any $j\in J$,
		 \begin{align}
		  \label{eq:707}
      &\overline{a}_{k_j}\in A^{\overline{\varepsilon}_{a,k_j}}(\overline{y}_{k_j}),
			        \quad  \overline{b}_{k_j}\in B^{\overline{\varepsilon}_{b,\,k_j}}(\overline{x}_{k_j}),\\
      \label{eq:618}
			&\gamma\norm{\overline{a}_{k_j}+\overline{b}_{k_j}}
			       =\norm{\overline{x}_{k_j}-\overline{y}_{k_j}}\leq \dfrac{2d_{0,\gamma}}{j},\\
			\label{eq:619}
			&\overline{\varepsilon}_{a,\,k_j}+\overline{\varepsilon}_{b,\,k_j}\leq
			\dfrac{2(1+\sigma/\sqrt{1-\sigma^2})d_{0,\gamma}^{\,2}}{\gamma j}\,.
		 \end{align}
	\end{enumerate}
\end{proposition}
\begin{proof}
Note first that \eqref{eq:706} follow from the inclusions
in \eqref{eq:err.b} and \eqref{eq:err.a}. 
Using the last statement in Proposition \ref{pr:eHPE}, Theorem \ref{lm:rhpe2} 
(with $\underline{\lambda}=1$) 
and \eqref{eq:def.d0} we obtain that there exists $i\in \{1,\dots, j\}$
such that
\begin{align}
 \norm{v_{k_i}}\leq \dfrac{d_{0,\gamma}}{\sqrt{j}}
			\sqrt{\dfrac{1+\sigma}{1-\sigma}},\quad  \varepsilon_{k_i}\leq \dfrac{\sigma^2 d_{0,\gamma}^2}{2(1-\sigma^2)j}\,,
\end{align}
which, in turn, combined with the identity in \eqref{eq:err.a} and the definitions of $v_{k_i}$ and $\varepsilon_{k_i}$
in \eqref{eq:pr:eHPE.02} gives the desired inequalities in \eqref{eq:306} and \eqref{eq:206} (concluding the
proof of (a)) and
\begin{align}
 \label{eq:6001}
     \norm{\overline{v}_j}\leq \dfrac{2d_{0,\gamma}}{j},\quad
     \overline{\varepsilon}_j\leq \dfrac{2(1+\sigma/\sqrt{1-\sigma^2})d_{0,\gamma}^2}{j}\,,
\end{align}
where $\overline{v}_j$ and $\overline{\varepsilon}_j$ are defined in \eqref{eq:d.eg2}
and \eqref{eq:d.eg3}, respectively, with $\Lambda_j=j$ and 
\begin{align}
 \label{eq:2021}
 \lambda_\ell:=1,\quad v_\ell:=v_{k_\ell},\quad \varepsilon_\ell:=\varepsilon_{k_\ell},
\quad \tilde z_\ell:=\widetilde z_{k_\ell}\qquad \forall \ell=1,\dots, j.
\end{align}

Since the inclusions in \eqref{eq:707} are a direct consequence of the ones in \eqref{eq:err.b}
and \eqref{eq:err.a}, Proposition \ref{pr:teps}(d), \eqref{eq:def.erg131}--\eqref{eq:def.erg133}
and Theorem \ref{th:tf}, it follows from \eqref{eq:618}, \eqref{eq:619}
and \eqref{eq:6001} that
to finish the proof of (b), it suffices to prove
that 
\begin{align}
 \label{eq:2022}
 \overline{v_j}=\gamma(\overline{a}_{k_j}+\overline{b}_{k_j}),\quad \gamma(\overline{a}_{k_j}+\overline{b}_{k_j})=\overline{x}_{k_j}-\overline{y}_{k_j},
\qquad \overline{\varepsilon}_j=
\gamma(\overline{\varepsilon}_{a,\,k_j}+\overline{\varepsilon}_{b,\,k_j}). 
\end{align}
The first identity in \eqref{eq:2022} follows from 
\eqref{eq:2021}, the second identities
in \eqref{eq:d.eg2} and \eqref{eq:pr:eHPE.02}, and \eqref{eq:def.erg130}. 
%
%
On the other hand, from \eqref{eq:err.a}
we have $\gamma(a_{k_\ell}+b_{k_\ell})=x_{k_\ell}-y_{k_\ell}$, for all $\ell=1,\dots, j$, which 
combined with \eqref{eq:def.erg131} and \eqref{eq:def.erg130} gives the second identity in \eqref{eq:2022}. Using
the latter identity and the second one in \eqref{eq:2022} we obtain
\begin{align}
 \label{eq:2023}
 (y_{k_\ell}-\overline{y}_{k_j})+\gamma(b_{k_\ell}-\overline{b}_{k_j})=(x_{k_\ell}-\overline{x}_{k_j})-\gamma(a_{k_\ell}-\overline{a}_{k_j})
\qquad \forall \ell=1,\dots, j.
\end{align}
Moreover, it follows from \eqref{eq:d.eg2}, \eqref{eq:2021}, 
the first identity in \eqref{eq:pr:eHPE.02}, \eqref{eq:def.erg131} 
and \eqref{eq:def.erg130} that
\begin{align}
 \label{eq:7001}
 \overline{\widetilde z}_{j}=\overline{\widetilde z}_{k_j}=\dfrac{1}{j}\sum_{\ell=1}^j\,\left(y_{k_\ell}+\gamma b_{k_\ell}\right)=
 \overline{y}_{k_j}+\gamma \overline{b}_{k_j}.
\end{align}
Using \eqref{eq:7001}, \eqref{eq:2021}, \eqref{eq:pr:eHPE.02} and \eqref{eq:2023} we obtain, for
all $\ell=1,\dots, j$,
\begin{align*}
 \nonumber
 \inner{\widetilde{z}_{\ell}-\overline{\widetilde z}_{j}}{v_{\ell}}&=
\inner{(y_{k_\ell}+\gamma b_{k_\ell})-(\overline{y}_{k_j}+\gamma \overline{b}_{k_j})}{\gamma(a_{k_\ell}+b_{k_\ell})}\\
  \nonumber 
	&=\gamma \inner{(y_{k_\ell}-\overline{y}_{k_j})+\gamma(b_{k_\ell}-\overline{b}_{k_j})}{a_{k_\ell}}+
	\gamma \inner{(y_{k_\ell}-\overline{y}_{k_j})+\gamma(b_{k_\ell}-\overline{b}_{k_j})}{b_{k_\ell}}\\
 \nonumber
&=\gamma \inner{(y_{k_\ell}-\overline{y}_{k_j})+\gamma(b_{k_\ell}-\overline{b}_{k_j})}{a_{k_\ell}}+
	\gamma \inner{(x_{k_\ell}-\overline{x}_{k_j})-\gamma(a_{k_\ell}-\overline{a}_{k_j})}{b_{k_\ell}}\\
	&=\gamma\inner{y_{k_\ell}-\overline{y}_{k_j}}{a_{k_\ell}}+\gamma^2\inner{b_{k_\ell}-\overline{b}_{k_j}}{a_{k_\ell}}
	+\gamma\inner{x_{k_\ell}-\overline{x}_{k_j}}{b_{k_\ell}}-\gamma^2\inner{a_{k_\ell}-\overline{a}_{k_j}}{b_{k_\ell}},
\end{align*}
which combined with \eqref{eq:d.eg3}, \eqref{eq:2021}, \eqref{eq:def.erg132} and \eqref{eq:def.erg133} yields
\begin{align*}
 \overline{\varepsilon}_j=\dfrac{1}{j}\sum_{\ell=1}^j\,
 \left[\varepsilon_\ell+\inner{\widetilde{z}_{\ell}-\overline{\widetilde z}_{j}}{v_{\ell}}\right]
 &= \dfrac{1}{j}\sum_{\ell=1}^j\,
 \gamma \left[\varepsilon_{b,\,k_\ell}+\inner{x_{k_\ell}-\overline{x}_{k_j}}{b_{k_\ell}}+\inner{y_{k_\ell}-\overline{y}_{k_j}}{a_{k_\ell}}\right]\\
 &=\gamma(\overline{\varepsilon}_{a,\,k_j}+\overline{\varepsilon}_{b,\,k_j}),
 \end{align*}  
which is exactly the last identity in \eqref{eq:2022}. This finishes the proof.
\end{proof}

\begin{proposition}{\bf (rate of convergence for null steps)}
 \label{pr:comp.b}
 Let $\{(x_k,b_k)\}$, $\{(y_k,a_k)\}$ and $\{\varepsilon_{b,k}\}$
be generated by \emph{Algorithm \ref{inexact.dr}}. Let $\{\beta_k\}$ and 
the set $\mathcal{B}$ be defined in \eqref{eq:def.betak} and \eqref{eq:def.ab}, respectively.
Then, for $k\in \mathcal{B}$,
 \begin{align}
  \label{eq:com.bb}
	  & a_k\in A(y_k),\quad b_k\in B^{\varepsilon_{b,\,k}}(x_k),\\
	  \label{eq:com.bb2}
	&\gamma\norm{a_k+b_k}=\norm{x_k-y_k}\leq \dfrac{2\sqrt{\tau_0}}{\sigma}\;\theta^{\frac{\beta_{k-1}}{2}},\\
		\label{eq:com.bb3}
	& \gamma \varepsilon_{b,\,k}\leq \dfrac{\tau_0}{2}\,\theta^{\beta_{k-1}}.
 \end{align}
\end{proposition}
\begin{proof}
Note first that \eqref{eq:com.bb} follows from \eqref{eq:err.b} and \eqref{eq:err.a}.
Using \eqref{eq:def.ab}, \eqref{eq:err.b} and Step 3.b's definition (see Algorithm \ref{inexact.dr})
we obtain
\begin{align*}
 \tau_{k-1}\geq \norm{\underbrace{\gamma b_k+x_k-z_{k-1}}_{p_k}}^2+2\gamma\varepsilon_{b,k}>
     \sigma^2\norm{\underbrace{\gamma b_k+y_k-z_{k-1}}_{q_k}}^2,
\end{align*}
which, in particular, gives
\begin{align}
 \label{eq:1301}
 \gamma \varepsilon_{b,\,k}\leq \dfrac{\tau_{k-1}}{2},
\end{align}
and combined with the identity in \eqref{eq:err.a} yields,
\begin{align}
 \nonumber
 \gamma\norm{a_k+b_k}=\norm{x_k-y_k}&=\norm{p_k-q_k}\\
\nonumber
 &\leq \norm{p_k}+\norm{q_k}\\ 
 \label{eq:ss34}
	&\leq \left(1+\dfrac{1}{\sigma}\right)\sqrt{\tau_{k-1}}.
\end{align}
To finish the proof, use \eqref{eq:1301}, \eqref{eq:ss34} and 
\eqref{eq:tau.beta}.
\end{proof}

Next we present the main results regarding the pointwise and ergodic iteration complexity of 
Algorithm \ref{inexact.dr} for finding approximate solutions of \eqref{eq:mip} satisfying the termination criterion
\eqref{eq:def.apsol}. While Theorem \ref{th:idr.main} is a consequence of Proposition \ref{pr:c.ihpe}(a) and Proposition \ref{pr:comp.b}, the ergodic iteration complexity of Algorithm \ref{inexact.dr}, namely Theorem \ref{th:idr.main.erg}, follows by combining the latter proposition and
Proposition \ref{pr:c.ihpe}(b). Since the proof of Theorem \ref{th:idr.main.erg} follows the same outline of 
Theorem \ref{th:idr.main}'s proof, it will be omitted.

\begin{theorem}\emph{{\bf (pointwise iteration complexity of Algorithm \ref{inexact.dr})}}
 \label{th:idr.main}
Assume that $\max\{(1-\sigma)^{-1},\sigma^{-1}\}=\mathcal{O}(1)$ and let $d_{0,\gamma}$ be as
in \eqref{eq:def.d0}. Then,  for given tolerances $\rho,\epsilon>0$, \emph{Algorithm \ref{inexact.dr}} finds $a,b,x,y\in \HH$ 
and $\varepsilon_b\geq 0$ such that
\begin{align}
 \label{eq:main.02}
 a\in A(y),\;b\in B^{\varepsilon_b}(x),\quad \gamma\norm{a+b}=\norm{x-y} \leq \rho,\quad 
\varepsilon_b\leq \epsilon
\end{align}
after performing \emph{at most}
\begin{align}
 \label{eq:ext.steps.alg3}
\mathcal{O}\left(1+\max\left\{\dfrac{d_{0,\gamma}^{\,2}}{\rho^2},
	  \dfrac{d_{0,\gamma}^{\,2}}{\gamma\epsilon}\right\} \right)
\end{align}
\emph{extragradient steps} and
\begin{align}
 \label{eq:null.steps.alg3}
 \mathcal{O}\left(1+\max\left\{\log^+\left(\dfrac{\sqrt{\tau_0}}{\rho}\right),\log^+\left(\dfrac{\tau_0}{\gamma \epsilon}\right)\right\}\right)
\end{align}
\emph{null steps}. As a consequence, under the above assumptions, \emph{Algorithm \ref{inexact.dr}} terminates with $a,b,x,y\in \HH$ 
and $\varepsilon_{b}\geq 0$ satisfying \eqref{eq:main.02} in at most
\begin{align}
 \label{eq:main.01}
	\mathcal{O}\left(1+\max\left\{\dfrac{d_{0,\gamma}^{\,2}}{\rho^2},
	  \dfrac{d_{0,\gamma}^{\,2}}{\gamma\epsilon}\right\}
		+
\max\left\{\log^+\left(\dfrac{\sqrt{\tau_0}}{\rho}\right),\log^+\left(\dfrac{\tau_0}{\gamma \epsilon}
\right)\right\}\right)
\end{align}
iterations. 
\end{theorem}
\begin{proof}
Let $\mathcal{A}$ be as in \eqref{eq:def.ab} and consider the cases:
\begin{align}
 \label{eq:def.m01}
 \#\mathcal{A}\geq M_{\text{ext}}:=\left\lceil \max\left\{\dfrac{2\,d_{0,\gamma}^{\,2}}{(1-\sigma)\rho^2},
	  \dfrac{\sigma^2d_{0,\gamma}^{\,2}}{2\gamma(1-\sigma^2)\epsilon}\right\} \right\rceil\quad \mbox{and}\quad
	\#\mathcal{A}<M_{\text{ext}}.
	\end{align}
In the first case, the desired bound \eqref{eq:ext.steps.alg3} on the number of extragradient steps
to find $a,b,x,y\in \HH$ and $\varepsilon_b\geq 0$ satisfying \eqref{eq:main.02}
follows from the definition of $J$ in \eqref{eq:car.a} and Proposition \ref{pr:c.ihpe}(a).

On the other hand, in the second case, i.e., $\#\mathcal{A}<M_{\text{ext}}$, the desired bound \eqref{eq:null.steps.alg3} 
is a direct consequence of Proposition \ref{pr:comp.b}. The last statement of the theorem follows from \eqref{eq:ext.steps.alg3}
and \eqref{eq:null.steps.alg3}.
\end{proof}

Next is the main result on the ergodic iteration complexity of Algorithm \ref{inexact.dr}. As mentioned before, its 
proof follows the same outline of Theorem \ref{th:idr.main}'s proof, now applying Proposition \ref{pr:c.ihpe}(b) instead
of the item (a) of the latter proposition.

\begin{theorem}\emph{{\bf (ergodic iteration complexity of Algorithm \ref{inexact.dr})}}
 \label{th:idr.main.erg}
For given tolerances $\rho,\epsilon>0$, under the same assumptions of \emph{Theorem \ref{th:idr.main}},  
\emph{Algorithm \ref{inexact.dr}} provides $a,b,x,y\in \HH$ and $\varepsilon_a, \varepsilon_b\geq 0$ such that
\begin{align}
 \label{eq:main.04}
 a\in A^{\varepsilon_a}(y),\;b\in B^{\varepsilon_b}(x),\quad \gamma \norm{a+b}=\norm{x-y}\leq \rho,\quad 
\varepsilon_a+\varepsilon_b\leq \epsilon.
\end{align}
after performing at most
\begin{align}
 \label{eq:ext.steps.alg3e}
\mathcal{O}\left(1+\max\left\{\dfrac{d_{0,\gamma}}{\rho},
	  \dfrac{d_{0,\gamma}^{\,2}}{\gamma\epsilon}\right\} \right)
\end{align}
\emph{extragradient steps} and
\begin{align}
 \label{eq:null.steps.alg3e}
 \mathcal{O}\left(1+\max\left\{\log^+\left(\dfrac{\sqrt{\tau_0}}{\rho}\right),\log^+\left(\dfrac{\tau_0}{\gamma \epsilon}\right)\right\}\right)
\end{align}
\emph{null steps}. As a consequence, under the above assumptions, \emph{Algorithm \ref{inexact.dr}} terminates with $a,b,x,y\in \HH$ 
and $\varepsilon_{a},\varepsilon_{b}\geq 0$ satisfying \eqref{eq:main.04} in at most
\begin{align}
 \label{eq:main.03}
	\mathcal{O}\left(1+\max\left\{\dfrac{d_{0,\gamma}}{\rho},
	  \dfrac{d_{0,\gamma}^{\,2}}{\gamma\epsilon}\right\}
		+
\max\left\{\log^+\left(\dfrac{\sqrt{\tau_0}}{\rho}\right),\log^+\left(\dfrac{\tau_0}{\gamma \epsilon}
\right)\right\}\right)
\end{align}
iterations.
\end{theorem}
\begin{proof}
The proof follows the same outline of Theorem \ref{th:idr.main}'s proof, now applying Proposition \ref{pr:c.ihpe}(b) instead
of Proposition \ref{pr:c.ihpe}(a).
\end{proof}

\noindent
{\bf Remarks.}
\begin{enumerate}
 \item Theorem \ref{th:idr.main.erg} ensures that for given tolerances $\rho, \epsilon>0$, up to an additive logarithmic factor, Algorithm \ref{inexact.dr} requires
no more than
\[
 \mathcal{O}\left(1+\max\left\{\dfrac{d_{0,\gamma}}{\rho},\dfrac{d_{0,\gamma}^{\,2}}{\gamma\epsilon}\right\}\right)
\]
iterations to find an approximate solution of the monotone inclusion problem \eqref{eq:mip} according to the termination criterion \eqref{eq:def.apsol}.
\item While the (ergodic) upper bound on the number of iterations provided in \eqref{eq:main.03} is better
than the corresponding one in \eqref{eq:main.01} (in terms of the dependence on the tolerance $\rho>0$) by a factor of 
$\mathcal{O}(1/\rho)$, the inclusion in \eqref{eq:main.04} is potentially weaker than the corresponding one in \eqref{eq:main.02}, since
one may have $\varepsilon_a>0$ in \eqref{eq:main.04}, and the set $A^{\varepsilon_a}(y)$ is in general 
larger than $A(y)$.
\item Iteration complexity results similar to the ones in Proposition \ref{pr:c.ihpe} were recently obtained for a relaxed Peaceman-Rachford method
in \cite{mon.che-com.17} . We emphasize that, in contrast to this work, 
the latter reference considers
only the case where the resolvents $J_{\gamma A}$ and $J_{\gamma B}$ of $A$ and $B$, respectively, are both computable. 
\end{enumerate}

The proposition below will be important in the next section.

\begin{proposition}
\label{pr:imp.fact2}
Let $\{z_k\}$ be generated by \emph{Algorithm \ref{inexact.dr}} and 
$d_{0,\gamma}$ be as in \eqref{eq:def.d0}. Then, 
\begin{align}
\label{eq:ineq.zk02}
\|z_k-z_0\| \le 2 d_{0,\gamma}\quad \forall k\geq 1.
\end{align}
\end{proposition}
\begin{proof}
Note that (i) if $k=k_j\in \mathcal{A}$, for some $j\in J$, see \eqref{eq:car.a}, then \eqref{eq:ineq.zk02}
follows from the last statement in Proposition \ref{pr:eHPE} and Proposition \ref{pr:imp.fact}; (ii) if 
$k\in \mathcal{B}$, from the first identity in \eqref{eq:null.step}, see \eqref{eq:def.ab}, we find that
 either $z_k=z_0$, in which case \eqref{eq:ineq.zk02} holds trivially, or $z_k=z_{k_j}$ for some $j\in J$, in which case
the results follows from (i).
 \end{proof}

\section{A Douglas-Rachford-Tseng's forward-backward (F-B) four-operator splitting method}
 \label{sec:drt}

In this section, we consider problem \eqref{eq:drti}, i.e., the problem of finding $z\in \HH$ such that
\begin{align}
 \label{eq:drt}
 0\in A(z)+C(z)+F_1(z)+F_2(z)
\end{align}
where the following hold:
\begin{itemize}
 \item[(E1)] $A$ and $C$ are (set-valued) maximal monotone operators on $\HH$. 
 \item[(E2)]  $F_1:D(F_1)\subset \HH\to \HH$ is monotone and $L$-Lipschitz continuous on a (nonempty) closed convex set $\Omega$
such that $D(C)\subset \Omega\subset D(F_1)$, i.e., $F_1$ is monotone on $\Omega$ and there exists $L\geq 0$ such that
 \begin{align}
  \label{eq:f.Lip}
  \norm{F_1(z)-F_1(z')}\leq L\norm{z-z'}\qquad \forall z,z'\in \Omega.
 \end{align}
 \item[(E3)] $F_2:\HH\to \HH$ is $\eta-$cocoercivo, i.e., there exists $\eta>0$ such that
 \begin{align}
  \label{eq:f.coco}
  \inner{F_2(z)-F_2(z')}{z-z'}\geq \eta\norm{F_2(z)-F_2(z')}^2\qquad \forall z,z'\in \HH.
 \end{align}
 \item[(E4)] $B^{-1}(0)$ is nonempty, where
 \begin{align}
 \label{eq:def.bcff}
 B:=C+F_1+F_2.
\end{align}
 \item[(E5)] The solution set of \eqref{eq:drt} is nonempty.
\end{itemize} 

Aiming at solving the monotone inclusion \eqref{eq:drt}, we present and study the iteration complexity of a (four-operator) splitting method
which combines Algorithm \ref{inexact.dr} (used as an outer iteration) and a Tseng's forward-backward (F-B) splitting type method 
(used as an inner iteration for solving, for each outer iteration, the prox subproblems in \eqref{eq:err.b}). We prove results
on pointwise and ergodic iteration complexity of the proposed four-operator splitting algorithm by analyzing it in the framework of Algorithm \ref{inexact.dr} for solving \eqref{eq:mip} with $B$ as in \eqref{eq:def.bcff} and under assumptions (E1)--(E5). The (outer)
 iteration complexities will follow from results on pointwise and ergodic iteration complexities of Algorithm \ref{inexact.dr}, 
obtained in Section \ref{sec:dr}, while the computation of an upper bound on the overall number of inner iterations required to achieve 
prescribed tolerances will require a separate analysis. 
Still regarding the results
on iteration complexity, we mention that we consider the following notion of approximate solution for \eqref{eq:drt}: 
given tolerances $\rho, \epsilon>0$, find $a,b,x,y\in \HH$ and $\varepsilon_a,\varepsilon_b\geq 0$ such that
\begin{align}
\nonumber
 &a\in A^{\varepsilon_a}(y),\\
 \label{eq:app.sol.acff2}
 &\mbox{either}\;\;b\in C(x)+F_1(x)+F_2^{\varepsilon_b}(x)\;\;\mbox{or}\;\;b\in \left(C+F_1+F_2\right)^{\varepsilon_b}(x),\\
\nonumber
&\gamma\norm{a+b}=\norm{x-y}\leq \rho,\;\; \varepsilon_a+\varepsilon_b\leq \epsilon,
\end{align} 
where $\gamma>0$. Note that (i) for $\rho=\epsilon=0$, the above conditions 
imply that $z^*:=x=y$ is a solution 
of the monotone inclusion \eqref{eq:drt}; (ii) the second inclusion in \eqref{eq:app.sol.acff2}, which will appear in the ergodic 
iteration complexity, is potentially weaker than the first one (see Proposition \ref{pr:teps}(b)), which will appear in the 
corresponding pointwise iteration complexity of
the proposed method.

We also mention that problem \eqref{eq:drt} falls in the framework of the monotone inclusion \eqref{eq:mip} due to the facts that, in view of assumptions  (E1), (E2) and (E3), 
the operator $A$ is maximal monotone, and the operator $F_1+F_2$ is monotone and $(L+1/\eta)$--Lipschitz continuous
on the closed convex set $\Omega\supset D(C)$, which combined with the assumption on the operator $C$ in (E1) and with \cite[Proposition A.1]{MonSva10-1}
implies that the operator $B$ defined in \eqref{eq:def.bcff} is maximal monotone as well. These facts combined with assumption (E5) give 
that 
conditions (D1) and (D2) of Section \ref{sec:dr} hold for $A$ and $B$ as
in (E1) and \eqref{eq:def.bcff}, respectively. In particular, it gives that Algorithm \ref{inexact.dr} may be applied 
to solve the four-operator monotone inclusion  \eqref{eq:drt}. 

In this regard, we emphasize that any implementation of 
Algorithm \ref{inexact.dr} will heavily depend
on specific strategies for solving each subproblem in \eqref{eq:err.b}, since $(y_k,a_k)$ required in \eqref{eq:err.a}
can be computed by using the resolvent operator of $A$, available in closed form in many important cases.   
In the next subsection, we show how the specific structure \eqref{eq:drt} allows for an application of a Tseng's F-B splitting
type method for solving each subproblem in \eqref{eq:err.b}.

\subsection{Solving the subproblems in \eqref{eq:err.b} for $B$ as in \eqref{eq:def.bcff}}
 \label{subsec:solsub}
In this subsection, we present and study 
a Tseng's F-B splitting type method~\cite{bau.com-book,arias.davis-half,MonSva10-1,tse-mod.sjco00} 
for solving the corresponding proximal subproblem in \eqref{eq:err.b} at each (outer) iteration of Algorithm \ref{inexact.dr},
when used to solve  \eqref{eq:drt}. To begin with, first 
consider the (strongly) monotone inclusion
\begin{align}
 \label{eq:cff} 
  0\in B(z)+\dfrac{1}{\gamma}(z-\bpt{z})
\end{align}
where $B$ is as in \eqref{eq:def.bcff}, $\gamma>0$ and $\bpt{z}\in \HH$, and note that 
the task of finding 
$(x_k,b_k,\varepsilon_{b,k})$ 
satisfying \eqref{eq:err.b} is related to the task of solving \eqref{eq:cff} with $\bpt{z}:=z_{k-1}$.  

In the remaining part of this subsection,
we present and study a Tseng's F-B splitting type method for solving \eqref{eq:cff}.
As we have mentioned before, the resulting algorithm will be used as an inner procedure for solving the subproblems \eqref{eq:err.b} at each iteration
of Algorithm \ref{inexact.dr}, when applied to solve \eqref{eq:drt}.

\mgap
\mgap

\noindent
\fbox{
\addtolength{\linewidth}{-2\fboxsep}%
\addtolength{\linewidth}{-2\fboxrule}%
\begin{minipage}{\linewidth}
\begin{algorithm}
\label{hff}
{\bf A Tseng's F-B splitting type method for \bf{(\ref{eq:cff})}} 
\end{algorithm}
{\bf Input:} $C, F_1,\Omega,L, F_2$ and $\eta$ as in conditions (E1)--(E5), $\bpt{z}\in \HH$, $\bpt{\tau}>0$, $\sigma\in (0,1)$ and 
$\gamma$ such that
\begin{align}
 \label{eq:bound.gamma}
 0<\gamma\leq \dfrac{4\eta \sigma^2}{1+\sqrt{1+16L^2\eta^2\sigma^2}}.
\end{align}

\begin{itemize}
\item[(0)] Set $z_0\leftarrow \bpt{z}$ and $j\leftarrow 1$.
\item [(1)] Let $z'_{j-1}\leftarrow P_{\Omega}(z_{j-1})$ and compute 
\begin{align}
\label{eq:hff}
 \begin{aligned}
  &\widetilde z_j=
	\left(\frac{\gamma}{2} C+I\right)^{-1}\left(\dfrac{\bpt{z}+z_{j-1}-\gamma(F_1+F_2)(z'_{j-1})}{2}\right),\\
	[1mm]
	&z_j=\widetilde z_j-\gamma\left(F_1(\widetilde z_j)-F_1(z'_{j-1})\right).
	\end{aligned}
\end{align} 
\item[(2)] If
 \begin{align}
  \label{eq:hff2}
    \hspace{-2.4cm}\norm{z_{j-1}-z_j}^2+\frac{\gamma\norm{z'_{j-1}-\widetilde z_j}^2}{2\eta}\leq \bpt{\tau},
 \end{align}
then {\bf terminate}. Otherwise, set $j\leftarrow j+1$ and go to step 1.
   \end{itemize}
{\bf Output:} $(z_{j-1},z'_{j-1},z_j,\widetilde z_j)$. 
\noindent
\end{minipage}
} 
\mgap
\mgap

\noindent
{\bf Remark.}
\begin{enumerate}
 \item []
Algorithm \ref{hff} combines ideas from the standard Tseng's F-B splitting algorithm~\cite{tse-mod.sjco00} as well as
from recent insights on the convergence and iteration complexity of some variants the latter 
method~\cite{alv.mon.sva-reg.siam16,arias.davis-half,MonSva10-1}. In this regard, evaluating 
the cocoercive component $F_2$ just once per iteration (see \cite[Theorem 1]{arias.davis-half}) is
potentially important in many applications, where the evaluation of
cocoercive operators is in general computationally expensive (see \cite{arias.davis-half} for a discussion).
Nevertheless, we emphasize that the results obtained in this paper regarding the analysis of Algorithm \ref{hff} do not
follow from any of the just mentioned references.
\end{enumerate}

\noindent
Next corollary ensures that Algorithm \ref{hff} always terminates with the desired output.

\begin{corollary}
 \label{cor:cota.linear}
 Assume that $(1-\sigma^2)^{-1}=\mathcal{O}(1)$ and let $d_{\bpt{z},b}$ denote
 the distance of $\bpt{z}$ to $B^{-1}(0)\neq \emptyset$. Then, \emph{Algorithm \ref{hff}} 
 terminates with the desired output after performing no more than
\begin{align}
  \label{eq:bound.alg4}
  \mathcal{O}\left(1+\log^+\left(\dfrac{d_{\bpt{z},\,b}}{\sqrt{\bpt{\tau}}}\right)\right)
 \end{align}
iterations.
\end{corollary}
\begin{proof}
 See Subsection \ref{sec:proof.drt}.
\end{proof}

\subsection{A Douglas-Rachford-Tseng's F-B four-operator splitting method}
 \label{subsec:drt}

In this subsection, we present and study the iteration complexity of 
the main algorithm in this work, for solving \eqref{eq:drt}, namely Algorithm \ref{drt}, which combines Algorithm \ref{inexact.dr}, 
used as an outer iteration, and Algorithm \ref{hff}, used as an inner iteration, for solving the 
corresponding subproblem in \eqref{eq:err.b}. 
Algorithm \ref{drt} will be shown to be a special instance of Algorithm \ref{inexact.dr}, for which pointwise and ergodic iteration complexity results are available in Section \ref{sec:dr}. Corollary \ref{cor:cota.linear} will be specially important to compute a bound on the total number
of inner iterations performed by Algorithm \ref{drt} to achieve prescribed tolerances.

\mgap
\mgap

\noindent
\fbox{
\addtolength{\linewidth}{-2\fboxsep}%
\addtolength{\linewidth}{-2\fboxrule}%
\begin{minipage}{\linewidth}
\begin{algorithm}
\label{drt}
{\bf A Douglas-Rachford-Tseng's F-B splitting type method for \bf{(\ref{eq:drt})}}
\end{algorithm}
\begin{itemize}
\item[(0)] Let $z_0\in \HH$, $\tau_0>0$ and $0<\sigma,\theta<1$ be given, let 
$C,F_1,\Omega, L, F_2$ and $\eta$ as in conditions
(E1)--(E5) and $\gamma$ satisfying condition \eqref{eq:bound.gamma}, and set $k\leftarrow 1$.
\item [(1)] Call Algorithm \ref{hff} with inputs  
$C,F_1,\Omega, L, F_2$ and $\eta$, $(\bpt{z},\bpt{\tau}):=(z_{k-1},\tau_{k-1})$, $\sigma$ and
$\gamma$
to obtain as output $(z_{j-1},z'_{j-1},z_j,\widetilde z_j)$, and set 
\begin{align}
 \label{eq:err.b2}
 x_k=\widetilde z_j,\qquad b_k=\dfrac{z_{k-1}+z_{j-1}-(z_j+\widetilde z_j)}{\gamma},\qquad
\varepsilon_{b,\,k}=\frac{\norm{z'_{j-1}-\widetilde z_j}^2}{4\eta}.
\end{align}

\item[(2)]  Compute $(y_k,a_k)\in \HH\times \HH$ such that
 \begin{align}
 \label{eq:err.a2}
 \hspace{-3.0cm} a_k\in A(y_k),\quad \gamma a_k+y_k=x_k-\gamma b_k.
\end{align}
\item[(3)]  (3.a) If 
 \begin{align}
 \label{eq:cond.err2}
 \hspace{-1.3cm}\norm{\gamma b_k+x_k-z_{k-1}}^2+2\gamma \varepsilon_{b,k}\leq \sigma^2\norm{\gamma b_k+y_k-z_{k-1}}^2,
\end{align}
\hspace{0.9cm}then
\begin{align}
 \label{eq:ext.step2}
z_k=z_{k-1}-\gamma(a_k+b_k),\quad \tau_k=\tau_{k-1}\qquad \mbox{[extragradient step]}.
\end{align}
(3.b) Else
\begin{align}
 \label{eq:null.step2}
 \hspace{-3.6cm} z_k=z_{k-1},\quad \tau_k=\theta\,\tau_{k-1}\qquad \mbox{[null step]}.
\end{align}
\item[(4)]
	Set $k\leftarrow k+1$ and go to step 1.
   \end{itemize}
\noindent
\end{minipage}
} 
\mgap
\mgap

In what follows we present the pointwise and ergodic iteration complexities of
Algorithm \ref{drt} for solving the four-operator monotone inclusion problem
\eqref{eq:drt}. The results will follow essentially from the corresponding ones
for Algorithm \ref{inexact.dr} previously obtained in Section \ref{sec:dr}. On the other hand,
bounds on the number of inner iterations executed before achieving prescribed tolerances
will be proved by using Corollary \ref{cor:cota.linear}.

We start by showing that Algorithm \ref{drt} is a special instance of Algorithm \ref{inexact.dr}.

\begin{proposition}
 \label{pr:5eh4}
  The triple $(x_k,b_k,\varepsilon_{b,\,k})$ in \eqref{eq:err.b2} satisfies
	condition \eqref{eq:err.b} in \emph{Step 1} of \emph{Algorithm \ref{inexact.dr}}, i.e., 
	\begin{align}
	 \label{eq:err.b3}
	 b_k\in C(x_k)+F_1(x_k)+F_2^{\varepsilon_{b,k}}(x_k)\subset B^{\varepsilon_{b,\,k}}(x_k),\quad \norm{\gamma b_k+x_k-z_{k-1}}^2+2\gamma\varepsilon_{b,\,k}\leq \tau_{k-1},
	\end{align}
	where $B$ is as in \eqref{eq:def.bcff}. As a consequence, \emph{Algorithm \ref{drt}} is a special instance of \emph{Algorithm \ref{inexact.dr}} for solving \eqref{eq:mip} with $B$ as in \eqref{eq:def.bcff}.
\end{proposition}
\begin{proof}
Using the first identity in \eqref{eq:851}, the definition of $b_k$ in \eqref{eq:err.b2} 
as well as the fact that $\bpt{z}:=z_{k-1}$ in Step 1 of Algorithm \ref{drt}
we find 
%
\begin{align}
 \label{eq:two.i}
 b_k=v_j-\dfrac{1}{\gamma}(\widetilde z_j-z_{k-1})=v_j-\dfrac{1}{\gamma}(\widetilde z_j-\bpt{z}).
\end{align}
Combining the latter identity with the second inclusion in \eqref{eq:pr:hff}, the second identity in \eqref{eq:851} and the definitions
of $x_k$ and $\varepsilon_{b,\,k}$ in \eqref{eq:err.b2} we obtain the first inclusion in \eqref{eq:err.b3}. The second desired inclusion
follows from \eqref{eq:def.bcff} and Proposition \ref{pr:teps}(b).
To finish the proof of \eqref{eq:err.b3}, note that from the first identity in \eqref{eq:two.i}, the definitions
of $x_k$ and $\varepsilon_{b,\,k}$ in \eqref{eq:err.b2}, the definition of $v_j$ in \eqref{eq:851}
and \eqref{eq:hff2} we have
\begin{align}
 \norm{\gamma b_k+x_k-z_{k-1}}^2+2\gamma\varepsilon_{b,\,k}
 =\norm{z_{j-1}-z_j}^2+\frac{\gamma\norm{z'_{j-1}-\widetilde z_j}^2}{2\eta}\leq \bpt{\tau}=\tau_{k-1},
\end{align}
which gives the inequality in \eqref{eq:err.b3}. The last statement of the proposition follows
from \eqref{eq:err.b3}, \eqref{eq:err.b}--\eqref{eq:null.step} and \eqref{eq:err.a2}--\eqref{eq:null.step2}.
\end{proof}

\begin{theorem}\emph{{\bf (pointwise iteration complexity of Algorithm \ref{drt})}}
 \label{th:general}
Let the operator $B$ and $d_{0,\gamma}$ be as
in \eqref{eq:def.bcff} and \eqref{eq:def.d0}, respectively, and
assume that $\max\{(1-\sigma)^{-1},\sigma^{-1}\}=\mathcal{O}(1)$. Let also $d_{0,b}$ be the distance
of $z_0$ to $B^{-1}(0)\neq\emptyset$. Then, for given tolerances $\rho,\epsilon>0$, the following
hold:
 \begin{itemize}

  \item[\emph{(a)}] \emph{Algorithm \ref{drt}} finds $a,b,x,y\in \HH$ and $\varepsilon_b\geq 0$ such that
	\begin{align}
	 \label{eq:opt.alg5}
	a\in A(y),\;b\in C(x)+F_1(x)+F_2^{\varepsilon_b}(x),\quad \gamma\norm{a+b}=\norm{x-y}\leq \rho,\; \varepsilon_b\leq \epsilon	 
	\end{align}	
  after performing no more than
	\begin{align}
 \label{eq:main.09}
   k_{\emph{p;\,outer}}:=\mathcal{O}\left(1+\max\left\{\dfrac{d_{0,\gamma}^{\,2}}{\rho^2},
	  \dfrac{d_{0,\gamma}^{\,2}}{\gamma\epsilon}\right\}
		+
\max\left\{\log^+\left(\dfrac{\sqrt{\tau_0}}{\rho}\right),\log^+\left(\dfrac{\tau_0}{\gamma \epsilon}
\right)\right\}\right)
\end{align}
outer iterations.
	\item[\emph{(b)}] Before achieving the desired tolerance $\rho,\epsilon>0$, each iteration of \emph{Algorithm \ref{drt}} performs at most
	 \begin{align}
	  \label{eq:inner.iter}
		 k_{\emph{inner}}:=\mathcal{O}\left(1+ \log^+\left(\dfrac{d_{0,\gamma}+d_{0,b}}{\sqrt{\tau_0}}\right)
		+
		\max\left\{\log^+\left(\dfrac{\sqrt{\tau_0}}{\rho}\right),\log^+\left(\dfrac{\tau_0}{\gamma \epsilon}\right)\right\}
		\right)
		%
	 \end{align}
	inner iterations; and hence evaluations of the $\eta$--cocoercive operator $F_2$.
 \end{itemize}
As a consequence, \emph{Algorithm \ref{drt}} finds $a,b,x,y\in \HH$ and $\varepsilon_b\geq 0$ satisfying
\eqref{eq:opt.alg5} after performing no more than $k_{\emph{p;\,outer}}\times  k_{\emph{inner}}$ inner iterations.
\end{theorem}
\begin{proof}
(a) The desired result is a direct consequence of the last statements in Proposition \ref{pr:5eh4} and Theorem \ref{th:idr.main}, and 
the inclusions in \eqref{eq:err.b3}.

(b) Using Step 1's definition and Corollary \ref{cor:cota.linear} we conclude that, at each iteration $k\geq 1$ of Algorithm \ref{drt}, 
the number of inner iterations 
is bounded by 
\begin{align}
  \label{eq:bound.alg42}
  \mathcal{O}\left(1+\log^+\left(\dfrac{d_{z_{k-1},\,b}}{\sqrt{\tau_{k-1}}}\right)\right)
 \end{align}
where $d_{z_{k-1},\,b}$ denotes the distance of $z_{k-1}$ to $B^{-1}(0)$. Now, using the last statements in
Propositions \ref{pr:5eh4} and \ref{pr:eHPE}, Proposition \ref{pr:imp.fact} and a simple argument
based on the triangle inequality we obtain
\begin{align}
 \label{eq:bound.alg43}
 d_{z_{k-1},b}\leq 2d_{0,\gamma}+d_{0,b}\qquad \forall k\geq 1.
\end{align}
%
By combining \eqref{eq:bound.alg42} and \eqref{eq:bound.alg43} and using \eqref{eq:tau.beta} we find that, at every iteration $k\geq 1$, the number of inner
iterations is bounded by
\begin{align}
 \mathcal{O}\left(1+\log^+\left(\dfrac{d_{0,\gamma}+d_{0,b}}{\sqrt{\theta^{\beta_{k-1}}\tau_0}}\right)\right)=
\mathcal{O}\left(1+\log^+\left(\dfrac{d_{0,\gamma}+d_{0,b}}{\sqrt{\tau_0}}\right)+\beta_{k-1}\right).
\end{align}
Using the latter bound, the last statement in Proposition \ref{pr:5eh4}, the bound on the number of null steps of Algorithm \ref{inexact.dr} given in Theorem \ref{th:idr.main}, and \eqref{eq:def.betak} we conclude that, before achieving the prescribed tolerance $\rho, \epsilon>0$,
each iteration Algorithm \ref{drt} performs at most the number of iterations given in  \eqref{eq:inner.iter}. This concludes the proof of (b). 

To finish the proof, note that the last statement of the theorem follows directly from (a) and (b).
\end{proof}

\begin{theorem}\emph{{\bf (ergodic iteration complexity of Algorithm \ref{drt})}}
 \label{th:general.e}
For given tolerances $\rho,\epsilon>0$, under the same assumptions of \emph{Theorem \ref{th:general}} the following
hold:
 \begin{itemize}

  \item[\emph{(a)}] \emph{Algorithm \ref{drt}} provides \,$a,b,x,y\in \HH$ and $\varepsilon_a,\varepsilon_b\geq 0$ such that
	\begin{align}
	 \label{eq:eopt.alg5}
	a\in A^{\varepsilon_a}(y),\;b\in \left(C+F_1+F_2\right)^{\varepsilon_b}(x),\quad \gamma\norm{a+b}=\norm{x-y}\leq \rho,\; 
	\varepsilon_a+\varepsilon_b\leq \epsilon	 
	\end{align}	
  after performing no more than
	\begin{align}
 \label{eq:emain.09}
   k_{\emph{e;\,outer}}:=\mathcal{O}\left(1+\max\left\{\dfrac{d_{0,\gamma}}{\rho},
	  \dfrac{d_{0,\gamma}^{\,2}}{\gamma\epsilon}\right\}
		+
\max\left\{\log^+\left(\dfrac{\sqrt{\tau_0}}{\rho}\right),\log^+\left(\dfrac{\tau_0}{\gamma \epsilon}
\right)\right\}\right)
\end{align}
outer iterations.
	\item[\emph{(b)}] Before achieving the desired tolerance $\rho,\epsilon>0$, each iteration of \emph{Algorithm \ref{drt}} performs at most
	 \begin{align}
	  \label{eq:einner.iter}
		 k_{\emph{inner}}:=\mathcal{O}\left(1+ \log^+\left(\dfrac{d_{0,\gamma}+d_{0,b}}{\sqrt{\tau_0}}\right)
		+
		\max\left\{\log^+\left(\dfrac{\sqrt{\tau_0}}{\rho}\right),\log^+\left(\dfrac{\tau_0}{\gamma \epsilon}\right)\right\}
		\right)
		\end{align}
	inner iterations; and hence evaluations of the $\eta$--cocoercive operator $F_2$.
 \end{itemize}
As a consequence, \emph{Algorithm \ref{drt}} provides $a,b,x,y\in \HH$ and $\varepsilon_b\geq 0$ satisfying
\eqref{eq:eopt.alg5} after performing no more than $k_{\emph{e;\,outer}}\times  k_{\emph{inner}}$ inner iterations.
\end{theorem}
\begin{proof}
The proof follows the same outline of Theorem \ref{th:general}'s proof.
\end{proof}

\subsection{Proof of Corollary \ref{cor:cota.linear}}
 \label{sec:proof.drt}

We start this subsection by showing that Algorithm \ref{hff} is a special instance of
Algorithm \ref{shpe} for solving the strongly monotone inclusion \eqref{eq:cff}.

\begin{proposition}
 \label{pr:hff}
 Let $\{z_j\}$, $\{z'_j\}$ and $\{\widetilde z_j\}$ be generated by \emph{Algorithm \ref{hff}} and
let the operator $B$ be as in \eqref{eq:def.bcff}. Define, 
 \begin{align}
  \label{eq:851}
  v_j:=\dfrac{z_{j-1}-z_j}{\gamma}, \qquad 
	\varepsilon_{j}:=\frac{\norm{z'_{j-1}-\widetilde z_j}^2}{4\eta},\quad \forall j\geq 1.
 \end{align}
 Then, for all $j\geq 1$,
\begin{align}
 \label{eq:pr:hff}
 &v_j\in (1/\gamma)(\widetilde z_j-\bpt{z})+C(\widetilde z_j)+F_1(\widetilde z_j)+F_2^{\,\varepsilon_{j}}(\widetilde z_j)
\subset (1/\gamma)(\widetilde z_j-\bpt{z})+B^{\,\varepsilon_{j}}(\widetilde z_j),\\
 \label{eq:pr:hff.2}
 & \norm{\gamma v_j+\widetilde z_j-z_{j-1}}^2+
 2\gamma \varepsilon_{j} \leq \sigma^2\norm{\widetilde z_j-z_{j-1}}^2,\\
 \label{eq:pr:hff.3}
 & z_j=z_{j-1}-\gamma v_j. 
\end{align}
As a consequence, \emph{Algorithm \ref{hff}} is a special instance of \emph{Algorithm \ref{shpe}} with $\lambda_j\equiv \gamma$ for solving
\eqref{eq:bmu} with $S(\cdot):=(1/\gamma)(\cdot-\bpt{z})$.
\end{proposition}
\begin{proof}
 Note that the first identity in \eqref{eq:hff} gives
 \begin{align*}
  \dfrac{z_{j-1}-\widetilde z_j}{\gamma}- F_1(z'_{j-1})\in 
	(1/\gamma)(\widetilde z_j-\bpt{z})+C(\widetilde z_j)+F_2(z'_{j-1}).
 \end{align*}
 Adding $F_1(\widetilde z_j)$ in both sides of the above identity and using the second 
and first 
identities in \eqref{eq:hff} and \eqref{eq:851}, respectively, we find
 \begin{align}
  v_j=\dfrac{z_{j-1}-z_j}{\gamma}\in 
	(1/\gamma)(\widetilde z_j-\bpt{z})+C(\widetilde z_j)+F_1(\widetilde z_j)+F_2(z'_{j-1}),
 \end{align}
which, in turn, combined with Lemma \ref{lm:coco} and the definition of $\varepsilon_j$ in \eqref{eq:851} proves
the first inclusion in \eqref{eq:pr:hff}. 
Note now that the second inclusion in \eqref{eq:pr:hff} is a direct consequence
of \eqref{eq:def.bcff} and Proposition \ref{pr:teps}(b). Moreover, \eqref{eq:pr:hff.3} is a direct consequence of
the first identity in \eqref{eq:851}.

To prove \eqref{eq:pr:hff.2}, note that from \eqref{eq:851}, the second identity in \eqref{eq:hff}, \eqref{eq:bound.gamma} 
and \eqref{eq:f.Lip}
we have
\begin{align*}
 \norm{\gamma v_j+\widetilde z_j-z_{j-1}}^2+
 2\gamma \varepsilon_{j}&=\gamma^2\norm{F_1(\widetilde z_j)-F_1(z'_{j-1})}^2+
\frac{\gamma\norm{z'_{j-1}-\widetilde z_j}^2}{2\eta}\\
  &\leq \left(\gamma^2 L^2+\dfrac{\gamma}{2\eta}\right)\norm{z'_{j-1}-\widetilde z_j}^2\\
	&\leq \sigma^2 \norm{z_{j-1}-\widetilde z_j}^2,
\end{align*}
which is exactly the desired inequality, where we also used the facts that
$z'_{j-1}=P_{\Omega}(z_{j-1})$, $\widetilde z_j\in D(C)\subset \Omega$ and that
$P_\Omega$ is nonexpansive. The last statement of the proposition follows
from \eqref{eq:pr:hff}--\eqref{eq:pr:hff.3}, \eqref{eq:cff}, \eqref{eq:shpe} and \eqref{eq:shpe2}.
\end{proof}

\noindent
{\bf Proof of Corollary \ref{cor:cota.linear}.}
 Let, for all $j\geq 1$, $\{v_j\}$ and $\{\varepsilon_j\}$ be defined in \eqref{eq:851}. 
 Using the last statement in Proposition \ref{pr:hff} and Proposition \ref{pr:3m} 
with $\mu:=1/\gamma$ and $\underline{\lambda}:=\gamma$ 
we find
\begin{align}
 \label{eq:cota.linear6}
 \norm{\gamma v_j}^2+2\gamma\varepsilon_j\leq \dfrac{((1+\sigma)^2+\sigma^2)(1-\alpha)^{j-1}\norm{\bpt{z}-z_\gamma^*}^2}{1-\sigma^2},
\end{align}
where $z_\gamma^*:=(S+B)^{-1}(0)$ with $S(\cdot):=(1/\gamma)(\cdot-\bpt{z})$, i.e., $z^*_\gamma=(\gamma B+I)^{-1}(\bpt{z})$. 
Now, using \eqref{eq:cota.linear6}, \eqref{eq:851} and Lemma \ref{lm:rb} we obtain
\begin{align}
 \label{eq:cota.linear2}
  \norm{z_{j-1}-z_j}^2+\frac{\gamma\norm{z'_{j-1}-\widetilde z_j}^2}{2\eta}\leq
\dfrac{((1+\sigma)^2+\sigma^2)(1-\alpha)^{j-1}d_{\bpt{z},\,b}^{\,2}}{1-\sigma^2},
\end{align}
which in turn combined with \eqref{eq:hff2}, after some direct calculations, gives \eqref{eq:bound.alg4}.

\section{Numerical experiments}
 \label{sec:num}

In this section, we perform simple numerical experiments on the family of (convex) constrained quadratic programming
problems
 \begin{align}
  \label{eq:cpro}
 \begin{aligned}
 & \mbox{minimize}\,\, \dfrac{1}{2} \inner{Qz}{z}+\inner{e}{z}\\
 & \mbox{subject to }Kz=0, z\in X,
\end{aligned} 
\end{align}
%
%
%
where $Q\in \R^{n\times n}$ is symmetric and either positive definite or positive semidefinite, 
$e=(1,\dots, 1)\in \R^n$, $K=(k_{j})\in \R^{1\times n}$, with $k_j\in \{-1,+1\}$ for all
$j=1,\dots, n$, and $X=[0,10]^n$ is a box in $\R^n$. Problem \eqref{eq:cpro} appears, for instance, in
support vector machine classifiers (see, e.g., \cite{cha.lin-lib.acm,dav-con.sjo15}). Here, $\inner{\cdot}{\cdot}$
denotes the usual inner product in $\R^n$.
A vector $z^*\in \R^n$ is a solution of \eqref{eq:cpro}, if and only if
it solves the MIP 
\begin{align}
  \label{eq:cpro.foc}
 0\in N_{\mathcal{M}}(z)+N_X(z)+Qz+e,
\end{align}
where $\mathcal{M}:=\mathcal{N}(K):=\{z\in \R^n\,|\,Kz=0\}$. Problem \eqref{eq:cpro.foc} is clearly a special
instance of \eqref{eq:drt}, in which
\begin{align}
  \label{eq:eman}
 A(\cdot):=N_{\mathcal{M}}(\cdot),\quad C(\cdot):=N_X(\cdot), \quad F_1(\cdot):=0\;\;\mbox{and}\;\;F_2(\cdot):=Q(\cdot)+e.
\end{align}
Moreover,  in this case, $J_{\gamma A}=P_{\mathcal{M}}$ and $J_{\gamma C}=P_X$.

In what follows, we analyze the numerical performance of the 
following three algorithms for solving the MIP \eqref{eq:cpro.foc}:
\begin{itemize}
 \item The Douglas-Rachford-Tseng's F-B splitting method (Algorithm \ref{drt} (ALGO 5)) proposed in Section \ref{sec:drt}. 
 We set $\sigma=0.99$, $\theta=0.01$, the operators $A$, $C$, $F_1$ and $F_2$ as in \eqref{eq:eman}, and $\Omega=\R^n$, $L=0$ and $\eta=1/(\sup_{\norm{z}\leq 1}\norm{Qz})$
(which clearly satisfy the conditions (E1)--(E5) of Section \ref{sec:drt}). We also have set 
$\gamma=2\eta\sigma^2$ (see \eqref{eq:bound.gamma}) and $\tau_0=\norm{z_0-P_X(z_0)+Qz_0}^3+1$.
 \item The relaxed forward-Douglas-Rachford splitting (rFDRS) from \cite[Algorithm 1]{dav-con.sjo15} 
(originally proposed in \cite{ari-for.opt15}). We set (in the notation of \cite{dav-con.sjo15}) 
$\beta_V=1/(\sup_{\norm{z}\leq 1}\norm{(P_{\mathcal{M}}\circ Q\circ P_{\mathcal{M}})z})$, $\gamma=1.99\beta_V$ and 
$\lambda_k\equiv 1$.
\item The three-operator splitting scheme (TOS) from \cite[Algorithm 1]{dav.yin-thr.svva17}. 
We set (in the notation of \cite{dav.yin-thr.svva17}) $\beta=1/(\sup_{\norm{z}\leq 1}\norm{Qz})$, $\gamma=1.99\beta$ and 
$\lambda_k\equiv 1$.
\end{itemize}
For each dimension $n\in \{100,500,1000,2000, 6000\}$, we analyzed the performance of each the above mentioned 
algorithms on a set of 100 randomly generated instances of \eqref{eq:cpro}. All the experiments were performed on a laptop equipped with an Intel i7 7500U CPU, 8 GB DDR4 RAM and a nVidia GeForce 940MX.
In order to allow performance comparison of ALGO 5, rFDRS and TOS, we adopted the stopping criterion
\begin{align}
  \label{eq:stc.impl}
 \norm{z_{k}-z_{k-1}}\leq 10^{-6},
\end{align}
\textbf{for which we considered only extragradient steps when analyzing the performance of \mbox{ALGO 5}}.
The corresponding experiments are displayed in Tables 1 ($Q$ positive definite), 2 (Table 1 continued), 3 ($Q$ positive semidefinite) and 4 (Table 3 continued).

Now note that by using \eqref{eq:err.a2} and \eqref{eq:ext.step2}, we conclude that \eqref{eq:stc.impl}
is equivalent to
\begin{align}
  \label{eq:stc.impl2}
 \gamma \norm{a_k+b_k}=\norm{x_k-y_k}\leq 10^{-6}.
\end{align}
Motivated by the above observation, we analyzed the performance of ALGO 5 on solving \eqref{eq:cpro} while using the
stopping criterion \eqref{eq:stc.impl2}, for which both extragradient and null steps are considered.  The corresponding results 
are displayed on Tables 5 and 6.

\mgap

Finally, we mention that \eqref{eq:cpro.foc} consists of a \emph{three-operator} MIP. 
For  future research, we intend to study the numerical performance of Algorithm \ref{drt} in  (true) \emph{four-operator}  MIPs. One possibility would be to consider structured minimization problems of the form
\begin{align}
  \label{eq:min4}
 \min_{x\in \HH}\,\left\{f(x)+g(x)+\varphi(\widetilde {K}x)+h(x)\right\}
\end{align}
where $f,g,\varphi:\HH\to (-\infty,+\infty]$ are proper closed convex functions, $h:\HH\to \R$ is convex and differentiable 
and $\widetilde K:\HH\to \HH$ is a bounded linear operator. Under certain qualification conditions, \eqref{eq:min4} is equivalent
to the MIP
\begin{align}
 0\in \partial f(x)+\partial g(x)+\widetilde{K}^*\partial \varphi(\widetilde{K} x)+\nabla h(x)
\end{align}
which, in turn, is clearly equivalent to
\begin{align}
  \label{eq:mip.pd}
 \begin{aligned}
 & 0\in \partial f(x)+\partial g(x)+ \widetilde{K}^*y+\nabla h(x)\\
& 0\in \partial \varphi^*(y)-\widetilde{K}x,
 \end{aligned}
\end{align}
where $\varphi^*$  denotes the Fenchel-conjugate of $\varphi$. We now note that \eqref{eq:mip.pd} is a special
instance of \eqref{eq:drt} where, by letting $z=(x,y)$,
\begin{align}
 \begin{aligned}
 &A(z):=\partial f(x)\times \partial\varphi^*(y),\quad C(z):=\partial g(x)\times \{0\},\\
 & F_1(z):=(\widetilde K^*y,-\widetilde Kx),\quad F_2(z):=(\nabla h(x),0).
 \end{aligned}
\end{align}
Hence, under mild conditions on \eqref{eq:min4} (specially regarding conditions (E1)--(E5) on Section \ref{sec:drt}), Algorithm \ref{drt} is
potentially applicable to solve \eqref{eq:mip.pd} (i.e., \eqref{eq:min4}).

We also mention that while the variational problem \eqref{eq:min4} appears in different applications in Imaging and related fields,
the primal-dual formulation \eqref{eq:mip.pd} has been widely used in 
nowadays research in designing efficient primal-dual methods
for, in particular,  solving \eqref{eq:min4} (see, e.g., \cite{bot.cse.hei-con.mp15,bot.cse.hen-rec.mwb14,cha.poc-fir.jmiv11,com-sys.jco13,com.pes-pri.svva12}).


\begin{table}[!htpb]
	\caption{Running time (in seconds) and number of iterations performed by ALGO 5, rFDRS and TOS
to reach the stopping criterion \eqref{eq:stc.impl} on a set of
100 randomly generated instances of \eqref{eq:cpro} with the matrix $Q$ \textbf{positive definite}, with
$n\in \{100,500,1000,2000,6000\}$. We can see that either ALGO 5 or TOS outperform the rFDRS in terms of (mean) running time, while ALGO 5 shows
a slightly superior performance on large dimensions. Moreover, -- see Table 2 -- when compared to TOS,  ALGO 5 provides a much more accurate approximate solution to the (unique) solution of \eqref{eq:cpro}.
}
	\begin{center} 
		\begin{tabular}{llcccp{0.25cm}rrr}
			\hline \\
		 & \multicolumn {3}{r}{ Time } & & & \multicolumn {3}{c}{ Iterations}  \\ 
		 \cline {3-5} \cline{7-9}\\
		 
		 $n$  & Algorithm & Min & Max & Mean &  & Min & Max & Mean  \\ \hline \\
		 
		      & ALGO 5 & 0.0014   & 0.0132   & 0.0017   & & 12 & 21 & 15.21  \\
		$100$ & rFDRS  & 0.0014   & 0.0170   & 0.0018   & & 6  & 15 & {\bf 8.31}  \\
		      & TOS    & 0.0010   & 0.0097   & \textbf{0.0012}   & & 7  & 16 & 9.37  \\ \hline
		      & ALGO 5 & 0.0302   & 0.0451   & \textbf{0.0334}   & & 15 & 22 & 17.24  \\
		$500$ & rFDRS  & 0.0517   & 0.0843   & 0.0626   & & 8  & 15 & {\bf 10.26}  \\
		      & TOS    & 0.0288   & 0.0434   & 0.0340   & & 9  & 16 & 11.28  \\ \hline
		      & ALGO 5 & 0.3177   & 0.4190   & \textbf{0.3540}   & & 14 & 20 & 17.14  \\
	   $1000$ & rFDRS  & 0.4768   & 0.7304   & 0.5771   & & 9  & 14 & {\bf 10.84}  \\
		      & TOS    & 0.3103   & 0.4504   & 0.3760   & & 10 & 15 & 11.86  \\ \hline
		      & ALGO 5 & 3.6411   & 3.9397   & \textbf{3.7648}   & & 19 & 21 & 19.80  \\
	   $2000$ & rFDRS  & 5.0062   & 5.5711   & 5.2795   & & 11 & 12 & {\bf 11.70}  \\
		      & TOS    & 3.6632   & 4.0798   & 3.7703   & & 12 & 13 & 12.70  \\ \hline
		      & ALGO 5 & 94.7551  & 121.1123 &\textbf{ 101.6311} & & 18 & 20 & 18.81  \\
	   $6000$ & rFDRS  & 107.3812 & 125.9018 & 115.0631 & & 11 & 13 & {\bf12.20}  \\
		      & TOS    & 94.7152  & 123.6527 & 104.0517 & & 12 & 15 & 13.18  \\
		 \hline\\
		      
		\end{tabular}
		
	\end{center}
	
  \end{table}
	
  \begin{table}[!htpb]
         \caption{Table 1 continued. Here (1) we show the number of extragradient and null steps performed by ALGO 5 while reaching the stopping criterion \eqref{eq:stc.impl}; (2) we evaluate the absolute error between the provided 
iterate $z_k$ and the unique solution $z^*$ of \eqref{eq:cpro}. We can see that, when compared to TOS, both ALGO 5 and rFDRS provide a much more accurate approximate solution.}
	\begin{center}
		\begin{tabular}{lllllp{0.1cm}p{0.5cm}p{0.5cm}p{0.5cm}p{0.1cm}rrr}
			\hline \\
			& & \multicolumn {3}{c}{Extragradient steps} & & \multicolumn {3}{c}{Null steps} & & \multicolumn {3}{c}{$ || z_{k}-z_{*} || $} \\ 
			\cline {3-5} \cline{7-9} \cline{11-13}\\
			
			$n$  & Algorithm & Min & Max & Mean &  & Min & Max & Mean &  & Min & Max & Mean \\ \hline 
	
			& ALGO 5 & 8 & 13 & 10.24 & & 3 & 4 & 3.58 & & 0.0033 & 0.7829 & {\bf 0.2714} \\
	  $100$ & rFDRS  &   &    &       & &   &   &      & & 0.0014 & 2.0768 & 0.3401 \\
			& TOS    &   &    &       & &   &   &      & & 2.8702 & 5.7959 & 4.2820 \\ \hline
			& ALGO 5 & 9 & 14 & 11.64 & & 3 & 5 & 4.30 & & 0.0009 & 1.0619 & \textbf{0.2004} \\
	  $500$ & rFDRS  &   &    &       & &   &   &      & & 0.0029 & 1.0275 & 0.3336 \\
			& TOS    &   &    &       & &   &   &      & & 5.7643 & 10.0704 & 7.8648 \\ \hline
			& ALGO 5 & 10 & 16 & 12.55 & & 4 & 5 & 4.53 & & 0.0036 & 0.5649 & \textbf{0.1798} \\
	 $1000$ & rFDRS  &    &    &       & &   &   &      & & 0.0008 & 0.9373 & 0.2763 \\
			& TOS    &    &    &       & &   &   &      & & 7.9120 & 13.3011 & 10.5356 \\ \hline
			& ALGO 5 & 11 & 19 & 13.30 & & 4 & 6 & 4.70 & & 0.0766 & 0.5499 & 0.2795 \\
	 $2000$ & rFDRS  &    &    &       & &   &   &      & & 0.1004 & 0.3433 & {\bf 0.2278} \\
			& TOS    &    &    &       & &   &   &      & & 13.3085 & 16.2610 & 14.4403\\ \hline
			& ALGO 5 & 13 & 18 & 15.20 & & 4 & 7 & 5.20 & & 0.0437 & 0.6245 & \textbf{0.2333}\\
	 $6000$ & rFDRS  &    &    &       & &   &   &      & & 0.1626 & 1.0021 & 0.4523 \\
			& TOS    &    &    &       & &   &  &       & & 19.9698 & 24.4657 & 22.9383 \\
			\hline\\
	       \end{tabular}
		\end{center}
	     \end{table}



\begin{table}[h!]
	
	\caption{Running time (in seconds) and number of iterations performed by ALGO 5, rFDRS and TOS
to reach the stopping criterion \eqref{eq:stc.impl} on a set of
100 randomly generated instances of \eqref{eq:cpro} with the matrix $Q$ \textbf{positive semidefinite}, with
$n\in \{100,500,1000,2000,6000\}$.
 Similarly to the case of $Q$ positive semidefinite, we can see that either ALGO 5 or TOS outperform the rFDRS in terms of (mean) running time, while ALGO 5 shows
a slightly superior performance on large dimensions.}
	
	\begin{center}
		 
		\begin{tabular}{llcccp{0.25cm}rrr}
			\hline \\
		 & \multicolumn {3}{r}{ Time } & & & \multicolumn {3}{c}{ Iterations}  \\ 
		 \cline {3-5} \cline{7-9}\\
		 
		 $n$  & Algorithm & Min & Max & Mean &  & Min & Max & Mean  \\ \hline \\
		 
		      & ALGO 5 & 0.0013   & 0.0130   & 0.0018  & & 11 & 20 & 15.31  \\
		$100$ & rFDRS  & 0.0013   & 0.0119   & 0.0017  & & 6  & 18 & {\bf 10.16}  \\
		      & TOS    & 0.0008   & 0.0056   & \textbf{0.0011}  & & 7  & 19 & 11.19  \\ \hline
		      & ALGO 5 & 0.0249   & 0.0409   & \textbf{0.0283}  & & 15 & 24 & 17.99  \\
		$500$ & rFDRS  & 0.0466   & 0.0895   & 0.0563  & & 9  & 18 & {\bf 12.13}  \\
		      & TOS    & 0.0236   & 0.0432   & 0.0291  & & 10 & 19 & 13.13  \\ \hline
		      & ALGO 5 & 0.3013   & 0.4061   & \textbf{0.3229}  & & 17 & 24 & 19.88 \\
	   $1000$ & rFDRS  & 0.4604   & 0.6780   & 0.5443  & & 10 & 16 & {\bf 13.18}  \\
		      & TOS    & 0.2910   & 0.4354   & 0.3348  & & 11 & 17 & 14.18  \\ \hline
		      & ALGO 5 & 3.6649   & 4.0284   & \textbf{3.7989}  & & 18 & 25 & 21.75  \\
	   $2000$ & rFDRS  & 5.0060   & 5.3364   & 5.1248  & & 13 & 15 & {\bf 14.20}  \\
		      & TOS    & 3.6933   & 4.0949   & 3.8607  & & 14 & 16 & 15.27  \\ \hline
		      & ALGO 5 & 101.0412 & 111.8105 & \textbf{107.9423} & & 20 & 23 & 21.80 \\
	   $6000$ & rFDRS  & 115.9101 & 146.0409 & 130.2121 & & 13 & 16 & {\bf 15.01}  \\
		      & TOS    & 105.1307 & 116.9615 & 110.3801 & & 14 & 17 & 16.05  \\
		 \hline
		\end{tabular}
		
	\end{center}
	
  \end{table}
  
  \begin{table}[h!]
  	
	\caption{Table 3 continued. Here we provide the number of extragradient and null steps performed by ALGO 5 while reaching the stopping criterion \eqref{eq:stc.impl}.}
	
	\begin{center}
		
		\begin{tabular}{llcccp{0.25cm}ccc}
			\hline \\
			& & \multicolumn {3}{c}{Extragradient steps} & & \multicolumn {3}{c}{Null steps}  \\ 
			\cline {3-5} \cline{7-9} \\
			
			$n$  & Algorithm & Min & Max & Mean & & Min & Max & Mean  \\ \hline \\
			
		 $100$	& ALGO 5  & 8  & 19 & 12.23 & & 3 & 5 & 3.77 \\
		 $500$	& ALGO 5  & 11 & 21 & 14.14 & & 4 & 5 & 4.31 \\
		$1000$	& ALGO 5  & 12 & 20 & 15.42 & & 4 & 5 & 4.54 \\
		$2000$	& ALGO 5  & 14 & 20 & 15.95 & & 4 & 6 & 4.70 \\
		$6000$	& ALGO 5  & 14 & 21 & 16.55 & & 4 & 7 & 5.35 \\ \hline\\
			
		\end{tabular}
		
	\end{center}
	
\end{table}



\begin{table}[h!]
	
	\begin{center}
		
	\caption{Running time (in seconds) and number of iterations performed by ALGO 5 to reach the stopping criterion \eqref{eq:stc.impl2} on a set of 100 randomly generated instances of \eqref{eq:cpro} with the matrix $Q$ \textbf{positive definite}, with
$n\in \{100,500,1000,2000,6000\}$. We can see a slight improvement when compared to the results obtained via the
stopping criterion \eqref{eq:stc.impl} -- cf. Table 1.}
		 
		\begin{tabular}{lcccp{0.25cm}rrr}
			\hline \\
		 & \multicolumn {3}{ c }{ Time } & \multicolumn {4}{c}{ Iterations}  \\ 
		 \cline { 2-4 } \cline{6-8}\\
		 $n$   & Min & Max & Mean & & Min & Max & Mean  \\ \hline \\
		$100$ & 0.0012  & 0.0121    & 0.0016 &  & 11 & 17 & 13.66 \\
		$500$ & 0.0251  & 0.0492    & 0.0314 & & 14 & 19 & 16.10\\
	   $1000$ & 0.3000  & 0.3539    & 0.3201 & & 15 & 20 & 17.32\\
	   $2000$ & 3.5538  & 3.7583    & 3.5914 & & 16 & 20 & 17.72 \\  
	   $6000$ & 98.8411 & 102.9118  & 99.7147& & 18 & 21 & 19.20 \\ \hline\\
		      
		\end{tabular}
		
	\end{center}
	
   \end{table}
	
	\begin{table}[h!]
		
	\begin{center}
		
		 \caption{Table 5 continued. Here, we show the number of extragradient and null steps performed by ALGO 5 while reaching the stopping criterion \eqref{eq:stc.impl2} and evaluate the absolute error between the provided 
iterate $z_k$ and the unique solution $z^*$ of \eqref{eq:cpro} -- cf. Table 2. }
		
		\begin{tabular}{llccp{0.25cm}cccp{0.25cm}rrr}
			\hline \\
			
			& \multicolumn {3}{ c }{Extragradient steps} & & \multicolumn {3}{c}{Null steps} & & \multicolumn {3}{c}{Absolute Error}  \\ 
			\cline { 2-4 } \cline{6-8} \cline{10-12} \\
			$n$   & Min & Max & Mean & & Min & Max & Mean & & Min & Max & Mean \\ \hline \\
		$100$ & 8  & 16 & 10.49 &  & 3 & 4 & 3.48 & & 0.0007 & 0.6949 & 0.2354\\
		$500$ & 9  & 15 & 11.82 &  & 3 & 5 & 4.22 & & 0.0007 & 0.7109 & 0.2134\\
	   $1000$ & 10 & 16 & 12.61 &  & 4 & 5 & 4.42 & & 0.0011 & 0.6628 & 0.1989\\
	   $2000$ & 12 & 16 & 13.46 &  & 4 & 6 & 4.64 & & 0.0061 & 0.4537 & 0.1596 \\  
	   $6000$ & 13 & 17 & 14.50 &  & 4 & 6 & 4.80 & & 0.0557 & 0.3666 & 0.2236 \\ \hline
			
		\end{tabular}
		
	\end{center}
\end{table}


\appendix
\section{Auxiliary results}


\begin{lemma}\emph{(\cite[Lemma 3.1]{alv.mon.sva-reg.siam16})}
 \label{lm:rb}
 Let $z_\gamma^*:=(\gamma B+I)^{-1}(\bpt{z})$ be the (unique) solution of \eqref{eq:cff}. Then,
\begin{align}
 \norm{\bpt{z}-z_\gamma^*}\leq \norm{\bpt{z}-x^*}\qquad \forall x^*\in B^{-1}(0).
\end{align}
\end{lemma}

\begin{lemma}\emph{(\cite[Lemma 2.2]{sva-cla.jota14})}
 \label{lm:coco}
 Let $F:\HH\to \HH$ be $\eta$--cocoercive, for some $\eta>0$, and let $z',\widetilde z\in \HH$. Then,
 \[
   F(z')\in F^{\varepsilon}(\widetilde z)\quad\mbox{where}\quad \varepsilon:=\dfrac{\norm{z'-\widetilde z}^2}{4\eta}.
 \]
\end{lemma}


\def\cprime{$'$}

\end{document}